\documentclass{smfart}

\usepackage[english, francais]{babel}
\usepackage[utf8]{inputenc}
\usepackage[T1]{fontenc}
\usepackage{graphicx}

\usepackage[version=0.96]{pgf}
\usepackage{tikz}
\usetikzlibrary{arrows,shapes,decorations,automata,backgrounds,petri}

\usepackage{algorithm}
\usepackage{algorithmic}

\usepackage{amsmath,amssymb,amsfonts,bbm}

\usepackage{float}
\usepackage{stmaryrd}

\usepackage[justification=centering]{caption}

\usepackage{hyperref}

\usepackage{diagbox}

\newcommand\tohide[1]{#1}

\newtheorem{thm}{Theorem}[section]
\newtheorem{prop}[thm]{Proposition}
\newtheorem{lemme}[thm]{Lemma}

\newtheorem{props}[thm]{Properties} 
\newtheorem{def-prop}[thm]{Definition-proposition}
\newtheorem{cor}[thm]{Corollary}

\newtheorem{rem}[thm]{Remark}
\newtheorem{conj}[thm]{Conjecture}

\newcommand{\Lb}{\mbox {\boldmath ${\Lambda}$}}
\newcommand{\Lam}{{\Lambda}}
\newcommand{\Dk}{{\mathcal D}}

\newcommand\Part{\mathcal P}
\newcommand\A{\mathcal A}
\newcommand\B{\mathcal B}

\newtheorem{ex}[thm]{Example}

\newcommand\N{\mathbb N}
\newcommand\Z{\mathbb Z}
\newcommand\Q{\mathbb Q}

\newcommand\R{\mathbb R}
\newcommand\C{\mathbb C}
\newcommand\T{\mathbb T}

\newcommand\transp[1]{\, {}^t \! {#1}}

\newcommand\abs[1]{\left|#1\right|}

\newcommand\floor[1]{\left\lfloor #1 \right\rfloor}
\newcommand\ceil[1]{\left\lceil #1 \right\rceil}

\newcommand\diam{\operatorname{diam}}

\renewcommand\B{\mathcal{B}}

\newcommand\defi[1]{\textbf{#1}}

\newcommand\Lrel{L^{\operatorname{rel}}}

\newcommand\trL{\mathstrut^t\!L}

\newcommand\Ab{\operatorname{Ab}}

\newcommand\cod{\operatorname{cod}}
\renewcommand\Re{\operatorname{Re}}
\renewcommand\Im{\operatorname{Im}}
\renewcommand\L{\mathbb{L}}

\renewcommand\S{\mathcal{S}}
\newcommand\set[2]{ \left\{ \! \! \! \begin{array}{l|r} #1 \, & \, #2 \end{array} \! \! \! \right\} }

\usepackage[all]{xy}

\author{Paul MERCAT}

\address{Aix-Marseille Université, 39 rue Frédéric Joliot-Curie 13453 Marseille cedex 13}

\email{paul.mercat@univ-amu.fr}

\urladdr{https://www.i2m.univ-amu.fr/~mercat.p/}

\author{Shigeki AKIYAMA}

\address{Institute of Mathematics, University of Tsukuba
1-1-1 Tennodai, Tsukuba, Ibaraki, 305-8571 Japan}

\email{akiyama@math.tsukuba.ac.jp}

\urladdr{http://math.tsukuba.ac.jp/~akiyama/}

\title{Yet another characterization of the Pisot substitution conjecture}

\begin{document}

\selectlanguage{english}

\frontmatter

\begin{abstract}
	We give a sufficient geometric condition for a subshift to be measurably isomorphic to a domain exchange and to a translation on a torus. 
	And for an irreducible unit Pisot substitution, we introduce a new topology on the discrete line and we give a simple necessary and sufficient condition for the symbolic system to have pure discrete spectrum. This condition gives rise to an algorithm based on computation of automata.
	To see the power of this criterion, we provide families of substitutions that satisfies the Pisot substitution conjecture: 
	1) $a \mapsto a^kbc$, $b \mapsto c$, $c \mapsto a$, for $k \in \N$ 
	and 
	2) $a \mapsto a^lba^{k-l}$, $b \mapsto c$, $c \mapsto a$, for $k \in \N_{\geq 1}$, for $0 \leq l \leq k$ using different methods.
	And we also provide an example of $\S$-adic system with pure discrete spectrum everywhere.
\end{abstract}



\date{\today}
\maketitle

\tableofcontents

\mainmatter

\section{Introduction}

Sturmian systems are well-known examples of subshifts that are conjugate to translations on the torus $\R/\Z$.
In 1982, Gérard Rauzy (see~\cite{rauzy}) gave a generalization to higher dimension for the subshift generated by the infinite fixed point of the Tribonnacci substitution:
\[
	\left\{
	\begin{array}{lcr}
		a &\mapsto& ab \\
		b &\mapsto& ac \\
		c &\mapsto& a
	\end{array}
	\right..
\]
He constructed a compact subset of $\R^2$, that we call now Rauzy fractal,
and that has the property that it tiles the plane by translation.
And we can define a domain exchange on this Rauzy fractal which is measurably conjugate to the subshift, and 
measurably conjugate to a translation on the two dimensional torus $\R^2/\Z^2$.

In 2001, P. Arnoux and S. Ito (see~\cite{AI}) generalized the work of Rauzy to any irreducible unit Pisot substitution.
They introduced a combinatorial condition which is easy to check, called the strong coincidence, that permits to get a measurable conjugacy between the subshift and a domain exchange, which is also a finite extension of a translation on a torus.

To obtain a measurable conjugacy between the subshift of an irreducible unit Pisot substitution and a translation on a torus,
several equivalent conditions (super coincidence, Geometric coincidence) has been studied (\cite{Ito-Rao:03, bk}).
This article gives another formulation of such coincidences 
and a short proof of its equivalence. 
The new criterion is checked by automata computation.

We introduce a topology on $\Z^d$ that permits to characterize easily when the subshift of a given irreducible unit Pisot substitution over $d$ letters is measurably isomorphic to a translation on a $(d-1)$-dimensional torus: see theorem~\ref{thm_int}.
And we show that this condition is equivalent to the non-emptiness of some computable regular language: see theorem~\ref{cint}.

In the last section, we use this condition to prove the pure discreteness for the family of substitution
\[
	s_k : \left\{ \begin{array}{l}
			a \mapsto a^kbc\\
			b \mapsto c\\
			c \mapsto a
		\end{array} \right.
\]
for $k \in \N$, where $a^k$ means that the letter $a$ is repeated $k$ times.
And we also prove the pure discreteness of the family of substitution
\[
	s_{l,k} : \left\{ \begin{array}{l}
			a \mapsto a^lba^{k-l}\\
			b \mapsto c\\
			c \mapsto a
		\end{array} \right.
\]
for $k \in \N_{\geq 1}$, $0 \leq l \leq k$, by computing explicitly a automaton describing algebraic relations, and showing that the pure discreteness for the substitution
\[
	s_k : \left\{ \begin{array}{l}
			a \mapsto a^kb\\
			b \mapsto c\\
			c \mapsto a
		\end{array} \right.
\]
implies the pure discreteness for the other substitutions.

We also use the criterion to prove, for all word in $\S^N$, the pure discreteness of the $\S$-adic system with
$\S = \{ \sigma, \tau \}$, where $\sigma$ and $\tau$ are the two substitutions
\[
        \sigma : \left\{ \begin{array}{ccc}
                a &\mapsto& aab \\
                b &\mapsto& c \\
                c &\mapsto& a
            \end{array} \right.
        \quad \text{ and } \quad
        \tau : \left\{ \begin{array}{ccc}
                a &\mapsto& aba \\
                b &\mapsto& c \\
                c &\mapsto& a
            \end{array}\right..
    \]

\section{A criterion for a subshift to have purely discrete spectrum}

In this section, we describe a general geometric criterion for a subshift to be measurably isomorphic to a translation on a torus.
Let us start by introduce some notations.

\subsection{Subshift}
	We denote by $A^\N$ (respectively $A^\Z$) the set of infinite (respectively bi-infinite) words over an alphabet $A$, and we denote by $A^* = \bigcup_{n \in \N} A^n$ the set of finite words over the alphabet $A$.
	We denote by $\abs{u}$ the length of a word $u$, and $\abs{u}_a$ denotes
	the number of occurrences of the letter $a$ in a word $u \in A^*$.
	And we denote by
	\[
		\Ab(u) = (\abs{u}_a)_{a \in A} \in \N^A
	\]
	the \defi{abelian vector} (or \defi{abelianisation}) of a word $u \in A^*$.
	The canonical basis of $\R^A$ will be denoted by $(e_a)_{a \in A} = (\Ab(a))_{a \in A}$.
	
	The \defi{shift} on infinite words is the application
	\[
		S :	\begin{array}{rcl}
					A^\N & \longrightarrow & A^\N \\
					(u_i)_{i \in \N} & \longmapsto & (u_{i+1})_{i \in \N}
				\end{array}
	\]
	We can also extend the shift on bi-infinite words in an obvious way, and it becomes invertible.
	For a word $u \in A^\N$, we denote $S^\N u = \set{S^n u}{n \in \N}$.
	
	We use the usual metric on $A^\N$:
	\[
		d(u,v) = 2^{-n} \text{ where $n$ is the length of the maximal common prefix}.
	\]
	
	The map $S$ is continuous for this metric.
	Given an infinite word $u$, the closure $\overline{S^\N u}$ is an $S$-invariant compact set. 
	We call \defi{subshift} generated by $u$, the dynamical system $(\overline{S^\N u}, S)$.
	
	The same can be done for bi-infinite words.
	
\subsection{Discrete line associated to a word} \label{ss_dl}
	Let $u \in A^\N$ be an infinite word over the alphabet $A$. 
	Then, the associated \defi{discrete line} is the following subset of $\Z^A$: 
	\[
		D_u := \set{ Ab(v) \in \Z^A }{ v \text{ finite prefix of } u }.
	\]
	If $u \in A^\Z$ is a bi-infinite word, then the corresponding discrete line is
	\[
		D_u := - D_v \cup D_w, 
	\]
	where $v,w \in A^\N$ are infinite words such that $u = \transp{v}w$, where $\transp{v} = ...v_n...v_2v_1$ denotes the mirror of the word $v = v_1v_2...v_n...$.
	
	For $u \in A^\N$, we can partition this discrete line into $d = \abs{A}$ pieces. For every $a \in A$, let
	\[
		D_{u,a} := \set{ Ab(v) \in \Z^A }{ v a \text{ finite prefix of } u }.
	\]
	The sets $D_{u,a} + e_a$, $a \in A$, also give almost a partition of $D_u$:
	\[
		D_{u} = \{ 0 \} \cup \bigcup_{a \in A} D_{u,a} + e_a. 
	\]
	For a bi-infinite word $u \in A^\Z$, we have the same, but we get a real partition, without the $\{0\}$.
	In both cases, these partitions permit to see the shift $S$ on the word $u$ as a domain exchange $E$:
	\[
		E:
		\begin{array}{rcl}
			D_{u} &\longrightarrow& D_{u} \\
			x & \longmapsto & x + e_a \text{ for } a \in A \text{ such that } x \in D_{u,a}.
		\end{array}
	\]
	
	There is also a property of tiling for this discrete line: we have the following
	\begin{prop} \label{conj_bef}
		Let $\Gamma_0$ be the subgroup of $\Z^A$ generated by $(e_a - e_b)_{a, b \in A}$, and let $u$ be any bi-infinite aperiodic word over the alphabet $A$.
		Then $D_u$ is a fundamental domain for the action of $\Gamma_0$ on $\Z^A$.
	Moreover the translation $T$ by $e_a$ (for any $a \in A$) on $\Z^A/\Gamma_0$ is conjugate to the domain exchange $E$ on $D_u$ by the natural quotient map $\pi_0: \Z^A \to \Z^A / \Gamma_0$, and the shift $(S^\Z u, S)$ is conjugate to the domain exchange $(D_u, E)$ by the map 
		\[
		    c: \begin{array}{ccc}
		            S^\Z u & \to & D_u \\
		            S^n u & \mapsto & E^n 0 
		    \end{array}.
		\]
	\end{prop}
	
	\begin{rem}
		We have the same for infinite non-eventually periodic words, but we get a fundamental domain for the action on the half-space
		\[
		    \set{ (x_a)_{a \in A} \in \Z^A }{ \sum_{a \in A} x_a \geq 0 },
        \]
        and a conjugacy with the shift on $S^\N u$.
	\end{rem}
	
	\begin{proof}
		The vectors $(e_a)_{a \in A}$ are equivalent modulo the group $\Gamma_0$. Hence, this discrete line is equivalent to $\Z e_a$ for any letter $a \in A$, and this is an obvious fundamental domain of $\Z^A$ for the action of $\Gamma_0$.
		The map $c$
		is well-defined and one-to-one because the word $u$ is aperiodic.
		And it gives a conjugacy between the shift $(S^\Z u, S)$ and the domain exchange $(D_u, E)$: $c \circ S = E \circ c$. 
		The natural quotient map $\pi_0: \Z^A \to \Z^A/\Gamma_0$ restricted to $D_u$ is bijective, and it gives a conjugacy between the domain exchange $(D_u, E)$ and the translation $(\Z^A/\Gamma_0, T)$: $\pi_0 \circ E = T \circ \pi_0$.
	\end{proof}
	
	If the discrete line $D_u$ stays near a given line of $\R^A$ (this will be the case for example for a periodic point of a Pisot substitution),
	then we can project onto a hyperplane $\Part$ of $\R^A$ (for example the hyperplane of equation $\sum_{a \in A} x_a = 0$) along this line. The projection of $\Z^A$ is dense in the hyperplane for almost all lines, and the group $\Gamma_0$ becomes a lattice in the hyperplane. If the projection of the discrete line is not so bad, we can expect that the closure gives a tiling of the hyperplane, and that the closure of each piece of the partition of the discrete line doesn't intersect each other.
	And we can expect that the conjugacy given by the previous proposition becomes a conjugacy of the closures.
	Figure~\ref{fig_idea} shows the conjugacy given by the proposition~\ref{conj_bef}, and what we get if everything goes well.
	
	\begin{figure}[h]
		\centering
		\caption{Commutative diagrams of the conjugacy between the shift $S$, the domain exchange $E$ and the translation on the quotient $T$, before and after taking the closure} \label{fig_idea}
		\begin{minipage}{.3\linewidth}
		\[
			\xymatrix
			{
				S^\Z u \ar[d]^c \ar[r]^S & S^\Z u \ar[d]^c \\
				D_u \ar[d]^{\pi_0} \ar[r]^E & D_u \ar[d]^{\pi_0} \\
				\Z^A / \Gamma_0 \ar[r]^T & \Z^A / \Gamma_0 \\
			}
		\]
		\end{minipage}
		$\leadsto$
		\begin{minipage}{.3\linewidth}
		\[
			\xymatrix
			{
				\overline{S^\Z u} \ar[d]^{\overline{c}} \ar[r]^S & \overline{S^\Z u} \ar[d]^{\overline{c}} \\
				\overline{\pi(D_u)} \ar[d]^{\pi_0} \ar[r]^E & \overline{\pi(D_u)} \ar[d]^{\pi_0} \\
				\Part / \pi(\Gamma_0) \ar[r]^T & \Part / \pi(\Gamma_0) \\
			}
		\]
		\end{minipage}
	\end{figure}

Let us now give a general geometric criterion that permits to know that everything works well as in Figure~\ref{fig_idea}.

\subsection{Geometrical criterion for the pure discreteness of the spectrum}

Here is the main general geometric criterion for a subshift to have a pure discrete spectrum. We use the notations defined in subsection~\ref{ss_dl}. 

\begin{thm} \label{thm1}
	Let $u \in A^\N$ be an infinite word over an alphabet $A$, and let $\pi$ be a linear projection from $\R^A$ onto a hyperplane $\Part$.
	We assume that we have the following:
	\begin{itemize}
		\item $\pi(\Z^A)$ is dense in $\Part$,
		\item the set $\pi(D_{u})$ is bounded,
		\item the subshift $(\overline{S^{\N} u}, S)$ is minimal,
		\item the boundaries of $\overline{\pi(D_{u,a})}$, $a \in A$, have zero Lebesgue measure,
		\item the union $\overline{\pi(D_u)} = \bigcup_{a \in A} \overline{\pi(D_{u,a})}$ is disjoint in Lebesgue measure.
	\end{itemize}
	
	Then there exists a $\sigma$-algebra and a $S$-invariant measure $\mu$ such that the subshift $(\overline{S^{\N} u}, S, \mu)$ is a finite extension of the translation of the torus $(\Part/\pi(\Gamma_0), T, \lambda)$, where $T$ is the translation by $\pi(e_a)$ (for any $a \in A$) on the torus $\Part / \pi(\Gamma_0)$, $\Gamma_0$ is the group generated by $\set{e_a - e_b}{a, b \in A}$, and $\lambda$ is the Lebesgue measure. And it is also a topological semi-conjugacy.

    If moreover the union
    \[
        \bigcup_{t \in \pi(\Gamma_0)} \overline{\pi(D_u)}+t = \Part
    \]
    is disjoint in Lebesgue measure,
	then the subshift $(\overline{S^\N u}, S, \mu)$
	is uniquely ergodic and
	is isomorphic
	to the translation on the torus $(\Part/\pi(\Gamma_0), T, \lambda)$ and to a domain exchange on $\overline{\pi(D_u)}$.
\end{thm}

\begin{rem}
	The disjointness in measure of the union
	\[
	    \bigcup_{t \in \pi(\Gamma_0)} \overline{\pi(D_u)}+t = \Part,
	\]
	implies the disjointness in measure of the union
	\[
	    \bigcup_{a \in A} \overline{\pi(D_{u,a})} = \overline{\pi(D_u)},
	\]
	and it also implies that the boundaries of $\overline{\pi(D_{u,a})}$, $a \in A$, have zero Lebesgue measure.
	
	Indeed, if we have $\lambda(\overline{\pi(D_{u,a})} \cap \overline{\pi(D_{u,b})}) > 0$, then we have
	\[
	    \lambda\left((\overline{\pi(D_u)} + \pi(e_a - e_b)) \cap \overline{\pi(D_u)}\right) \geq \lambda\left(\left[\overline{\pi(D_{u,a})} \cap \overline{\pi(D_{u,b})}\right] + \pi(e_a)\right)  > 0,
	\]
	so we have $a=b$.
	
	And we obtain that the boundary of each $\overline{\pi(D_{u,a})}$,\ $a \in A$, has zero Lebesgue measure, since
	\[
	    \partial \overline{\pi(D_{u,a})} \subseteq \overline{\pi(D_{u,a})} \cap \left( \bigcup_{b \in A \backslash \{a\}} \overline{\pi(D_{u,b})} \cup \bigcup_{t \in \pi(\Gamma_0) \backslash \{ 0\}} \overline{\pi(D_u)} + t \right).
	\]
\end{rem}

Before giving a proof of the theorem~\ref{thm1}, we confirm the unique ergodicity of a translation on a torus.

\begin{lemme}
    \label{UniqueErgodicity}
    Let $G$ be a compact group. The left multiplication action $x\mapsto gx$ is uniquely ergodic (w.r.t the Haar measure) if and only if the orbit $g^n (n=1,2,\dots)$ is dense in $G$.  
    \end{lemme}

\proof
See Theorem 4.14 in \cite{Einsiedler-Ward:11}.  
\qed
\medskip

Unique ergodicity of $(\Part/\pi(\Gamma_0), x\rightarrow x+\pi(e_a))$ follows from this lemma. Indeed as
$e_b\ (b\in A)$ are all equivalent mod $\Gamma_0$, the orbit 
$x\rightarrow x+\pi(e_a)$ on $\Part/\pi(\Gamma_0)\simeq \T^{d-1}$ 
is dense since $\pi(\Z^{A})$ is dense in $\Part$, where $d = \abs{A}$.

\begin{rem} \label{rem_dense}
Let $(x_2,\dots,x_{d})$ be the $d-1$ dimensional coordinates of $\pi(e_a)$ 
in the base $\{\pi(e_a-e_b) |\ b\neq a\}$. This defines a 
homomorphism $\Z \rightarrow \Part/\pi(\Gamma_0)\simeq \T^{d-1}$ 
of locally compact groups by 
$$
n \mapsto n (x_2,\dots,x_d).
$$
In light of Pontryagin duality, the denseness of $\{ n \pi(e_a) |\ n\in \N\}$ 
in $\Part/\pi(\Gamma_0)$ is equivalent 
to the injectivity of the dual map $\widehat{\T^{d-1}}\rightarrow \widehat{\Z}$. It is easy to confirm that this is also 
equivalent to the fact that
$1,x_1,x_2,\dots x_{d-1}$ are linearly independent over $\Q$ (see \cite[Chapter 2, Lemma 2]{Meyer}, \cite[Chapter 1, Theorem 25]{Siegel}).
\end{rem}

\begin{rem}
    The denseness of $\pi(\Z^A)$ in $\Part$ also implies that the restriction of $\pi$ to $\Z^A$ is injective.
\end{rem}

\subsubsection{Proof of the theorem~\ref{thm1}}

In order to prove this theorem, we start by showing that we can extend by continuity the map $\pi \circ c: S^\N u \to \pi(D_u)$ that gives the conjugacy between the shift $(S^\N u, S)$ and the domain exchange $(\pi(D_u), E)$.

\begin{lemme} \label{lc}
	Let $u \in A^\N$ be a non-eventually periodic infinite word over an alphabet $A$, and let $\pi$ be a projection from $\R^A$ onto a hyperplane $\Part$.
	We assume that $\pi(D_u)$ is bounded.
	Then the map
	\[
	    \pi \circ c: \begin{array}{ccc}
    	        S^\N u  &\to&       \pi(D_u) \\
    	        S^n u   &\mapsto&   E^n 0
    	   \end{array}
	\]
	can be extended by continuity at any point of the closure whose orbit is dense in $\overline{S^\N u}$.
\end{lemme}

To prove this lemma, we need the following geometric lemma, saying that we can always translate a bounded set of $\R^A$ in order to have a non empty but arbitrarily small intersection with the initial set.

\begin{lemme} \label{lt}
	Let $\Omega$ be a bounded and non-empty subset of $\R^A$.
	Then, we have
	\[
		\inf_{t \in \Omega - \Omega} \diam(\Omega \cap (\Omega - t)) = 0.
	\]
\end{lemme}

The proof is left as an exercise. It can be proven for example by considering a diameter and using the parallelogram law.


%

\begin{proof}[proof of lemma~\ref{lc}]
	
	Let $w \in \overline{S^\N u}$ having dense orbit in $\overline{S^\N u}$ and let $\epsilon > 0$. By lemma~\ref{lt}, there exists $t \in D_u - D_u$ such that $\diam(\pi(D_u) \cap (\pi(D_u) - \pi(t))) \leq \epsilon$.
	Let $n_1$ and $n_2 \in \N$ such that $c(S^{n_2} u) - c(S^{n_1} u) = t$. We can assume that $n_1 \leq n_2$ up to replace $t$ by $-t$.
	Then, there exists $n_0 \in \N$ such that $d(S^{n_0} w, u) \leq 2^{-n_2}$. Now, for all $v \in S^\N u$ such that $d(w, v) \leq 2^{-(n_0+n_2)}$, we have that $c(S^{n_0+n_1} v) \in D_u \cap (D_u - t)$, because $c(S^{n_0+n_2} v) - c(S^{n_0+n_1} v) = t$.
	Hence, if we let $\eta = 2^{-(n_0+n_2)}$, we have for all $v, v' \in D_u$,
	\[
			\left\{\begin{array}{c} d(v, w) \leq \eta \\ \text{ and } \\ d(v', w) \leq \eta \end{array}\right\} 
			\Longrightarrow d(\pi \circ c(v), \pi \circ c(v')) = d(\pi \circ c(S^{n_0+n_1} v), \pi \circ c(S^{n_0 + n_1} v')) \leq \epsilon.
	\]
	This proves that we can extend $\pi \circ c$ by continuity at point $w$.
\end{proof}

\begin{lemme} \label{lc2}
	Let $u \in A^\N$ be an infinite non-eventually periodic word over an alphabet $A$, and let $\pi$ be a projection from $\R^A$ onto a hyperplane $\Part$.
	We assume that we have the following conditions:
	\begin{itemize}
		\item $\pi(\Z^A)$ is dense in $\Part$,
		\item the set $\pi(D_u)$ is bounded,
		\item for every $a \in A$, the boundary of $\overline{\pi(D_{u,a})}$ has zero Lebesgue measure,
		\item the union $\bigcup_{a \in A} \overline{\pi(D_{u,a})} = \overline{\pi(D_u)}$, is disjoint in Lebesgue measure. 
    \end{itemize}
	Then the natural coding $\cod$ of $(\pi(D_u), E)$ for the partition $D_u = \bigcup_{a \in A} D_{u,a}$, can be extended by continuity to a full measure part $M$ of the closure. And we have
	\[
		\forall x \in M,\ \lim_{\substack{y \to x \\ y \in \pi(D_u)}} (\pi \circ c)^{-1}(y) = \cod(x).
	\]
\end{lemme}

\begin{proof}
	Let $\Omega = \overline{\pi(D_{u})}$ and $\forall a \in A,\ \Omega_a = \overline{\pi(D_{u,a})}$.
	Then the union $\bigcup_{a \in A} \overset{\circ}{\Omega_a}$ is disjoint, and we can extend the domain exchange $E$:
	\[
		E':
		\begin{array}{ccl}
			\bigcup_{a \in A} \overset{\circ}{\Omega_a} &\longrightarrow& \Omega \\
			x & \longmapsto & x + \pi(e_a) \text{ for } a \in A \text{ such that } x \in \overset{\circ}{\Omega_a}.
		\end{array}
	\]
	The part of full Lebesgue measure that we consider is the $E'$-invariant set
	\[
		M := \bigcap_{n \geq 1} {E'}^{-n} \Omega.
	\]
	
	Let $\epsilon > 0$ and let $x \in M$.
	Let $n_0 \in \N_{\geq 1}$ such that $2^{-n_0} \leq \epsilon$.
	The set
	\[
		M_{n_0} := \bigcap_{n = 1}^{n_0} {E'}^{-n} \Omega
	\]
	is an open set containing $x$, because $E'$ is continuous and $E'^{-1} \Omega = \bigcup_{a \in A} \overset{\circ}{\Omega_a}$ is open.
	Hence there exists $\eta > 0$ such that $B(x, \eta) \subseteq M_{n_0}$.
	And for every $y \in B(x, \eta) \cap M$, the natural coding of $(M, E')$ for the partition $M = \bigcup_{a \in A} M \cap \Omega_a + \pi(e_a)$ coincides with the coding of $x$ for the $n_0$ first steps.
	Hence, $\cod$ is continuous on $M$. We get also the last part of the lemma by observing that if $y \in B(x, \eta) \cap \pi(D_u)$, then the coding of $y$ (which is equal to $(\pi \circ c)^{-1}(y)$) also coincide with the coding of $x$ for the $n_0$ first steps.
\end{proof}


Now we can prove the main theorem of this section.
We start by extending the map $\pi \circ c$ by continuity, and we show that this map is almost everywhere one-to-one.
It gives us an isomorphism between the subshift $(\overline{S^\N u}, S, \mu)$, for some measure $\mu$, and a domain exchange defined Lebesgue-almost everywhere on $\overline{\pi(D_u)}$.
Then, we show that the map $\pi_0 : \overline{\pi(D_u)} \to \Part/\pi(\Gamma_0)$ is finite-to-one, and it gives us that the subshift is a finite extension of the translation on the torus $(\Part/\pi(\Gamma_0), T, \lambda)$.
Then if we assume that we have also the last hypothesis that $\overline{\pi(D_u)}$ tiles the hyperplane $\Part$, then we deduce
that we have the isomorphism with the translation on the torus, and we show that the unique ergodicity of the translation on the torus implies the unique ergodicity of the subshift.

\begin{proof}[proof of the theorem~\ref{thm1}]
    The hypothesis on the projection $\pi$ show that $u$ cannot be eventually periodic.
    Indeed, if $u$ was eventually periodic with a period $v \in A^*$, then the hypothesis that $\pi(D_u)$ is bounded implies that $\pi(\Ab(v)) = 0$, but this contradict the injectivity of the restriction of $\pi$ to $\Z^A$, thus it contradict the hypothesis that $\pi(\Z^A)$ is dense in $\Part$. 
    
	The lemma~\ref{lc} shows that we can extend the map $\pi \circ c$ by continuity to a map $\overline{c} : \overline{S^\N u} \to \overline{\pi(D_u)}$.
	If we compose $\overline{c}$ with the natural projection $\pi_0$ onto the torus $\Part/\pi(\Gamma_0)$, we get a continuous function which is onto, because of the equality $\pi(\Gamma_0) + \overline{\pi(D_u)} = \Part$ that comes from $\Gamma_0 + D_u = \Z^A$.
	And we have the equality
	\[
	    \pi_0 \circ \overline{c} \circ S = T \circ \pi_0 \circ \overline{c},
	\]
	where $T$ is the translation by $\pi(e_a)$ (for any $a \in A$) on the torus $\Part/\pi(\Gamma_0)$.
	Indeed, this equality is true on the dense subset $S^{\N} u$ by the proposition~\ref{conj_bef},
	and the maps $\pi_0$, $S$ and $T$ are continuous.
	This proves that the translation on the torus $(\Part/\pi(\Gamma_0), T)$ is a topological factor of the subshift $(\overline{S^\N u}, S)$.
	
	Let us consider the $\sigma$-algebra that we get from the Borel $\sigma$-algebra with the continuous map $\pi_0 \circ \overline{c}: \overline{S^\N u} \to \Part / \pi(\Gamma_0)$. A measure $\mu$ on this $\sigma$-algebra can be defined by $\mu((\pi_0 \circ \overline{c})^{-1}(A)) = \lambda(A)$ for any Borel set $A$ of $\Part / \pi(\Gamma_0)$, where $\lambda$ is the Lebesgue measure.
	By continuity, this measure $\mu$ that we get on $\overline{S^\N u}$ is $S$-invariant, and for this measure the translation of the torus $(\Part / \pi(\Gamma_0), T, \lambda)$ is a factor of the subshift $(\overline{S^\N u}, S, \mu)$.
	
	Then, the lemma~\ref{lc2} gives
	\[
	    \forall x \in \overline{c}^{-1}(M),\ x = \lim_{\substack{y \to x \\ y \in S^\N u}} (\pi \circ c)^{-1} \circ \pi \circ c (y) = \cod \circ \overline{c}(x).
	\]
	So the map $\overline{c}$ is one-to-one on the subset of full $\mu$-measure $\overline{c}^{-1}(M)$.
	Hence, the map $\overline{c}: \overline{S^\N u} \to \overline{\pi(D_u)}$ is a measurable conjugacy between the subshift $(\overline{S^\N u}, S, \mu)$ and the domain exchange $(\overline{\pi(D_u)}, E, \lambda)$.
	
	To prove that the subshift is a finite extension of the translation on the torus, it remains to show that the number of preimages by $\pi_0$ is bounded and almost everywhere constant. 
	The boundedness is a consequence of the hypothesis that $\pi(D_u)$ is bounded, and because $\pi(\Gamma_0)$ is a discrete subgroup of $\Part$.
	But this number of preimages is also decreasing by the translation, so by ergodicity of the translation on the torus, it is almost everywhere constant. Hence we get that the domain exchange $(\overline{\pi(D_u)}, E, \lambda)$ (which is isomorphic to the subshift $(\overline{S^\N u}, S, \mu)$) is a finite extension of the translation on the torus $(\Part/\pi(\Gamma_0), T, \lambda)$.
	
	%
	
	If we assume moreover that the sets $\overline{\pi(D_{u})} + t$, $t \in \pi(\Gamma_0)$, are disjoint in measure,
	then the map $\pi_0: \overline{\pi(D_u)} \to \Part / \pi(\Gamma_0)$ is invertible almost everywhere, is one-to-one on $\overline{c}^{-1}(M)$, and is a measurable conjugacy between the domain exchange $(\overline{\pi(D_u)}, E, \lambda)$ and the translation on the torus $(\Part / \pi(\Gamma_0), T, \lambda)$. And the map $\pi_0 \circ \overline{c}: \overline{S^\N u} \to \Part / \pi(\Gamma_0)$ is a measurable conjugacy between the subshift $(\overline{S^\N u}, S, \mu)$ and the translation on the torus $(\Part/\pi(\Gamma_0), T, \lambda)$.
	
	Then, the unique ergodicity of the translation on the torus implies that the subshift is also uniquely ergodic.
	Indeed, if $\mu'$ is an $S$-invariant measure of $\overline{S^\N u}$, then the pushforward $(\pi_0 \circ \overline{c})_* (\mu')$ is a $T$-invariant measure of the torus $\Part/\pi(\Gamma_0)$, so it is proportional to the Lebesgue measure.
	Then the $\mu'$-measure of the complementary of the set $\overline{c}^{-1}(M)$ is equal to the $(\pi_0 \circ \overline{c})_* (\mu')$-measure of the complementary of $\pi_0(M)$ which is equal to $0$. 
	And the restriction of $\pi_0 \circ \overline{c}$ to $\overline{c}^{-1}(M)$ is injective and bi-continuous, thus the measures $\mu'$ and $\mu$ are the same up to a scaling constant. 
\end{proof}

\subsection{An easy example : generalization of Sturmian sequences}
	An easy example where all works fine is obtained by taking a random line of $\R^d$ with a positive direction vector.
	We consider the natural $\Z^d$-tiling by hypercubes, and we take the sequence of hyperfaces that intersect the line.
	Almost surely, this gives a discrete line corresponding to some word $u$ over the alphabet of the $d$ type of hyperfaces.
	It is not difficult to see that the orthogonal projection $\pi$ along the line onto a hyperplane $\Part$ behave correctly for almost every choice of line.
	It gives a set whose closure tiles the plane, on which a domain exchange acts. This dynamics is conjugate to the subshift generated by the word $u$.
	It is also conjugate to the translation by $\pi(e_1)$ on the torus $\Part / \pi(\Gamma_0) \simeq \T^{d-1}$.	
	Figure~\ref{fig_exchange_cube} shows the domain exchange for a line whose a direction vector is around $\left(0.54973,\,0.36490,\,0.99501\right)$ in $\R^3$.
	
	\begin{figure}[h]
	\centering
		\caption{Domain exchange conjugate to a translation on the torus $\T^2$, and also conjugate to the subshift generated by a word corresponding to a discrete approximation of a line of $\R^3$. } \label{fig_exchange_cube}
		\tohide{
	        \includegraphics[scale=.25]{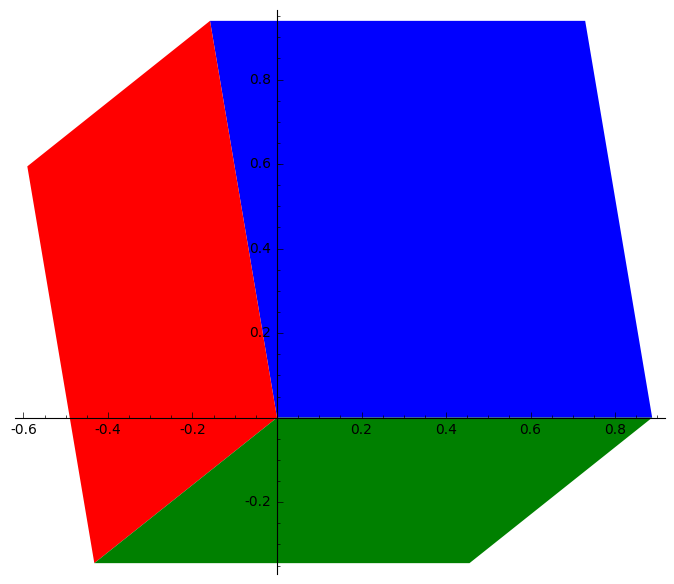}
		    \includegraphics[scale=.25]{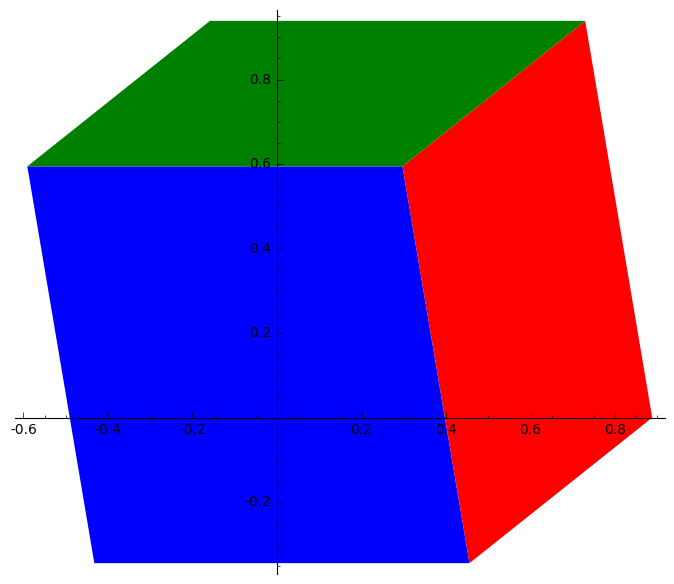}
		}
	\end{figure}
	
	\begin{rem}
		The word appearing in this last example is obtained by a simple algorithm: If the positive direction vector of the line is $(v_1, v_2, v_3)$, and if the line goes through the point $(c_1, c_2, c_3)$ then we have almost surely
		\begin{eqnarray*}
			&& \exists (k_1, k_2, k_3) \in \Z^3,\ \forall j \in \{1 ,2 ,3 \},\ k_j = \floor{v_j \frac{k_{i_n} - c_{i_n}}{v_{i_n}} + c_j} \\
			&\Longrightarrow& \exists ! i_{n+1} \in \{1, 2, 3\},\ \forall j \in \{1, 2, 3 \} \backslash \{ i_{n+1} \},\ k_j = \floor{v_j \frac{k_{i_{n+1}} + 1 - c_{i_{n+1}}}{v_{i_{n+1}}} + c_j}.
		\end{eqnarray*}
		The sequence $(i_n)_{n \in \N}$ define an infinite word over the alphabet $\{1,2,3\}$, and we get a bi-infinite word by invertibility of this algorithm.
	\end{rem}


\section{Pure discreteness for irreducible unit Pisot substitutions}

In this section, we define a topology on $\N^A$ that permits to give a simple condition to get the pure discreteness of the spectrum of the subshift coming from an irreducible Pisot unit substitution, using the criterion of the previous section.
And in the next section, we show that the reciprocal is true.

\subsection{Substitutions}

	Let $s$ be a substitution (i.e. a word morphism) over a finite alphabet $A$ of cardinality $d$.
	Let $M_s$ (or simply $M$ when there is no ambiguity) be the \defi{incidence matrix} of $s$. It is the $d \times d$ matrix whose coefficients are
	\[
		m_{a,b} = \abs{s(b)}_a,\ \forall (a,b) \in A^2,
	\]
	where $\abs{u}_a$ denotes the number of occurrences of the letter $a$ in a word $u \in A^*$.
	A \defi{periodic point} of $s$ is a fixed point of some power of $s$.
	It is an infinite word $u \in A^\N$ such that there exists $k \in \N_{\geq 1}$ such that $s^k(u) = u$. 
	
	A substitution $s$ is \defi{primitive} if there exists a $n \in \N$ such that for all $a$ and $b \in A$, the letter $b$ appears in $s^n(a)$.
	A substitution $s$ is \defi{irreducible} if the characteristic polynomial of its incidence matrix is irreducible, or equivalently if the degree of the Perron eigenvalue of the matrix equals the number of letters of the substitution.
	
	If $u$ is a periodic point of a primitive substitution, we can check that the subshift $(\overline{S^\N u}, S)$ depends only on the substitution and is minimal.
		
	We say that a substitution $s$ is \defi{Pisot} if the maximal eigenvalue of its incidence matrix is a Pisot number -- i.e. an algebraic integer greater than one, and whose conjugates have modulus less than one.
	If a substitution is Pisot irreducible, we can verify that the projection $\pi$ onto a hyperplane, along the eigenspace for the Pisot eigenvalue is such that $\pi(D_u)$ is bounded for a periodic point $u$ of the substitution.
	We say that a Pisot number is an \defi{unit} if its inverse is an algebraic integer.
	We say that a substitution is an \defi{irreducible Pisot unit substitution} if the substitution is irreducible (i.e. the characteristic polynomial of the incidence matrix is irreducible), the highest eigenvalue of the incidence matrix is a Pisot number, and the determinant of the incidence matrix is $\pm1$. It is equivalent to say that the incidence matrix has only one eigenvalue of modulus greater or equal to one, and that this eigenvalue is a Pisot unit number.
\subsection{Topology and main criterion} \label{topo}
	Let $A$ be a finite set, $\Part$ be an hyperplane of $\R^A$ (for example the hyperplane of equation $\sum_{a \in A} x_a = 0$), and $\pi$ be a linear projection along the expanding eigenvector of a matrix $M \in M_A(\Z)$.
	Note that this assumption implies $M \pi^{-1}(0) =\pi^{-1}(0)$.
	We define, for any subset $S$ of $\Part$, the \defi{discrete line} of points that project to $S$:
	\[
		Q_{S} = \set{ x \in \N^A }{ \pi(x) \in S }.
	\]
	This permits to define a topology on $\N^A$ by taking the following set of open sets
	\[
		\set{ Q_U }{ U \text{ open subset of } \Part }.
	\]
	
	\begin{rem}
		In a similar way, it is possible to define a topology on $\Z^A$ that takes care of bi-infinite words. 
	\end{rem}
	
	We extend the notion of interior to parts of $\Z^A$:
	\[
	    \forall P \subseteq \Z^A,\ \overset{\circ}{P} := \bigcup \set{U \text{ open set of $\Part$}}{Q_U \subseteq P}.
	\]
	
	\begin{props}
		The topology that we just defined has the following properties:
		\begin{itemize}
			\item If the projection $\pi$ is such that $\pi(\Z^A)$ is dense in $\Part$, then for any open subset $U$ of $\Part$, we have that $\pi(Q_U)$ is dense in $U$, and we have
				\[
					Q_U = \emptyset \Longleftrightarrow U = \emptyset.
				\]
			\item For any subset $S \subseteq \Part$ and any $t \in \Z^A$, the symmetric difference
				\[
					(Q_{S} + t) \Delta Q_{S + \pi(t)}
				\]
				is finite.
				In particular, we have $\overset{\circ}{Q_S} = \emptyset \Longleftrightarrow \overset{\circ}{Q}_{S+\pi(t)} = \emptyset \Longleftrightarrow \overset{\circ}{\overbrace{Q_S + t}} = \emptyset$.
			\item If $\det(M) \in \{-1, 1\}$, then for any subset $S$ of $\Part$,
					the symmetric difference $(M Q_S) \Delta Q_{\pi(M S)}$ is finite.
					In particular, we have $\overset{\circ}{\overbrace{M Q_S}} = \emptyset \Longleftrightarrow \overset{\circ}{Q}_{\pi(MS)} = \emptyset$.
			\item The space $\N^A$ is a Baire space for this topology.
		\end{itemize}
	\end{props}
	
	The fact that $\N^A$ is a Baire space follows from the fact that $\Part$ is a Baire space, by the Baire category theorem. Indeed, if $Q$ is a dense open set of $\N^A$, then there exists a dense open set $U$ of $\Part$ such that $Q = Q_U$. Hence, a countable intersection of dense open subsets of $\N^A$ is a dense subset of $\N^A$.
	
	This topology gives a necessary and sufficient condition for the subshift of a Pisot irreducible substitution, to have a pure discrete spectrum: 
	
	\begin{thm} \label{thm_int}
	    Let $s$ be an irreducible Pisot unit substitution over an alphabet $A$, and let $u \in A^\N$ be a periodic point of $s$.
	    Then, the subshift $(\overline{S^\N u}, S)$ has pure discrete spectrum if and only if
		$
			\exists a \in A,\ 
			\overset{\circ}{D}_{u,a} \neq \emptyset.
		$
	\end{thm}
	
	
	
	
	
	
%

\subsection{Proof that an inner point implies the pure discreteness of the spectrum}
	In this subsection, we prove the first statement and the sufficiency of the second statement. The necessity is proven in the next section. 
	
	\begin{proof}[Proof of the direct implication of theorem~\ref{thm_int}]
		Up to replace the substitution $s$ by a power, we can assume that the periodic point $u$ is a fixed point: $s(u)=u$.
		Let us show that the hypothesis of the theorem~\ref{thm1} are satisfied.
		\begin{itemize}
			\item $\pi(\Z^A)$ is dense in $\Part$: it is a consequence of the remark~\ref{rem_dense}, because the hypothesis that the characteristic polynomial of the matrix of the substitution is irreducible gives that the coefficients of an eigenvector of such matrix are linearly independant over $\Q$. 
			\item The set $\pi(D_u)$ is bounded: it is well known that for any Pisot unit irreducible substitution, the Rauzy fractal $\overline{\pi(D_u)}$ is compact.
			\item The subshift $(\overline{S^\N u}, S)$ is minimal: this is true for every primitive substitution, see~\cite{queff} proposition 5.5. 
		\end{itemize}
			Now, if we assume that $\exists a \in A,\ \overset{\circ}{D}_{u,a} \neq \emptyset$, then we have the following 
			
	    		\begin{lemme} \label{lp}
					We have 
					for all $b \in A$, $\overset{\circ}{D}_{u,b} \neq \emptyset$ and $\overline{D_{u,b}} = \overline{\overset{\circ}{D}}_{u,b}$.
				\end{lemme}
				
				\begin{proof}
					For every $b \in A$, we have the equality
					\[
						D_{u,b} = \bigcup_{c \xrightarrow{t} b \in \A^s} M D_{u,c} + t,
					\]
					where $c \xrightarrow{t} b \in \A^s$ means that it is a transition in the automaton $\A^s$ (i.e. there exists words $u,v \in A^*$ such that $s(c) = ubv$ and $\Ab(u) = t$). 
    					By primitivity, up to iterate enough this equality, the set $D_{u,a}$ appears in the union, so the set $D_{u,b}$ has non-empty interior as soon as one of them has non-empty interior.
					
					If we iterate $n+1$ times the equality, we get
					\[
						D_{u,b} = \bigcup_{c \xrightarrow{t_n} ... \xrightarrow{t_0} b \in \A^s} M^n D_{u,c} + \sum_{k=0}^n M^k t_k,
					\]
					where $c \xrightarrow{t_n} ... \xrightarrow{t_0} b \in \A^s$ means that there exists states $(q_i)_{i=0}^{n+1}$ of $\A^s$ with $q_0 = b$ and $q_{n+1} = c$, such that for every $i$, $q_{i+1} \xrightarrow{t_i} q_{i}$ is a transition in $\A^s$.
					Each term of this union has non-empty interior, because each $D_{u,c}$ has non-empty interior and $\det(M) \in \{-1,1\}$.
					And the diameter of each $\pi(M^n D_{u,c})$ tends to zero as $n$ tends to infinity, so it proves that the interior of $D_{u,b}$ is dense in $D_{u,b}$.
				\end{proof}

%
%
%
				
				
				Hence, $\overset{\circ}{D}_{u}$ is a dense open subset of $\overline{\pi(D_u)}$.
				By Baire's theorem, for all $t \in \Gamma_0 \backslash \{ 0 \}$, the empty set $\overset{\circ}{D_{u}} \cap (\overset{\circ}{D_{u}} + t)$ is a dense subset of $\overset{\circ}{\overline{D_{u}}} \cap (\overset{\circ}{\overline{D_{u}}} + t)$, therefore the sets $\overset{\circ}{\overline{D_{u}}}$ and $(\overset{\circ}{\overline{D_{u}}} + t)$ are disjoint. This gives the wanted disjointness in measure since the boundary has zero Lebesgue measure. 
				

%

		The hypothesis of the theorem~\ref{thm1} are satisfied,
		thus the subshift $(\overline{S^\N u}, S)$ is uniquely ergodic and measurably conjugate to the translation by $\pi(e_a)$ on the torus $\Part/\pi(\Gamma_0)$ with respect to the Lebesgue measure.
		In particular, it has pure discrete spectrum.
	\end{proof}

\section{Algebraic coincidence ensures an inner point}

In this section, we prove that pure discreteness of the subshift $(\overline{S^{\N}u},S)$ ensures the non emptyness of $\overset{\circ}{D_{u,a}}$ for some $a$.

\subsection{Algebraic coincidence of substitutive Delone set}
A Delone set is a 
relatively dense and uniformly discrete subset of $\R^A$. We say
that $\Lb=(\Lam_a)_{a\in A}$ is a {\em Delone multi-color set}
in $\R^A$ if each $\Lam_a$ is a Delone set and
$\cup_{a\in A} \Lam_a \subset \R^A$ is Delone. 
Here a `multi-set' $\Lb$ is simply a vector whose entries are Delone sets. 
We introduce this concept instead of taking their union, 
only because $\Lam_a \cap \Lam_b$ may not be empty for $a\neq b$. 
We think that each element of $\Lam_a$ has color $a$.
A set $\Lam \subset \R^A$ is a {\em Meyer set}
if it is a Delone set and there exists a finite set $F$ such that $\Lam-\Lam\subset \Lam+F$.
A Delone set is a Meyer set if and only if 
$\Lam - \Lam$ is uniformly discrete in $\R^A$ \cite{Lagarias:96}.
Note that a Meyer set has finite local complexity (FLC), i.e., for any $r>0$ 
there are only finitely many transitionally inequivalent clusters (configurations of points) 
in a ball of radius $r$.
$\Lb = (\Lam_a)_{a\in A}$ is called a {\em
substitution Delone multi-color set} if $\Lb$ is a Delone multi-color set and
there exist an expansive matrix
$B$ and finite sets $\Dk_{ab}$ for $a,b\in A$ such that
\begin{equation}
\label{eq-sub}
\Lambda_a = \bigcup_{b\in A} (B \Lambda_b + \Dk_{ab}),\ \ \ a \in A,
\end{equation}
where the union on the right side is disjoint. 
We consider that two Delone sets are close if they agree around a large ball centered at the origin up to a small translation. 
This defines a topology on the set of Delone sets, which is 
called {\bf local topology}. The closure of all translates of a
multi-color FLC Delone set is compact, and the translation action
gives a topological dynamical system. This is minimal and
uniquely ergodic if the substitution matrix $(^{\#}\Dk_{ab})_{(a,b) \in A^2}$ is primitive.
Lagarias and Wang \cite{Lagarias-Wang:03}
proved that $|\det B|$ must be equal to the Perron Frobenius root of
the substitution matrix.

To obtain our converse statement, we employ the knowledge on self-affine tiling dynamical systems. This is the minimal and uniquely ergodic
topological dynamical system generated by a self-affine tiling together with translation action (Solomyak \cite{Solomyak:97}).
One can rewrite this dynamics by a multi-colored Delone set $\Lb$ (see \cite{LMS}) in a natural manner: 
a point in $\Lam_a$ represents a tile colored by $a$ and 
the points are 
located in relatively the same position in the same colored tile.
We study the translation dynamical system of the multi-colored Delone set. 
Dynamical properties of a self-affine tiling dynamical system
are identical with the translation action 
of the corresponding multi-color Delone set.
In this setting, Lee \cite{Lee:07} introduced  {\it algebraic coincidence} of substitutive multi-color Meyer 
set in $\R^A$ which is equivalent to pure discreteness of the corresponding tiling dynamical system. 

In this section, we prove that if a one-dimensional substitutive Meyer set associated to 
an irreducible Pisot unit substitution satisfy the algebraic coincidence, then
there exists $a\in A$ such that $\stackrel{\circ}{D_{u,a}}\neq \emptyset$, which completes the proof of the main theorem. 

\subsection{Substitutive Meyer set from $D_u$}

Let $s$ be a primitive substitution over an alphabet $A$ whose
substitution matrix is $M$.
To study the spectral property of the word generated by $s$, we may replace $s$ by its suitable power. Thus without loss of generality, we may assume that $s$ has a bi-infinite fixed point. Consequently we 
assume that $v=\dots v_n\dots v_1$ and $u=u_1\dots u_n\dots$ are one-sided 
infinite words such that $vu\in A^{\Z}$ is a 2-sided fixed point of $s$. 
This means
the word $u$ (resp. $v$) is a right (resp. left) infinite fixed point of $s$ and $v_1u_1$
is a subword of $\sigma^n(a)$ for some $n\in \N$ and $a\in A$.
The shift $S((u_i))=(u_{i+1})$ is naturally extended to bi-infinite words and we obtain 
a topological dynamical system $(\overline{S^{\Z}vu},S)$ in a similar manner and the map $S$ is a homeomorphism.
Two systems $(\overline{S^{\N}u},S)$ and $(\overline{S^{\Z}vu},S)$ are the same up to countable exceptional points: Proposition 5.13 and Corollary 5.8 in \cite{queff} shows that spectral properties of them are identical.
The left abelianisation $D_{v,a}$ is defined by
$$
\left\{\left. -\sum_{i=1}^n e_{v_i} \right|\ v_n=a \right\}.
$$
and $D_{vu,a}=D_{v,a} \cup D_{u,a}$. As $u$ is a fixed point of a substitution, 
$D_{u,a}$ (resp. $D_{v,a}$) has one to one correspondence to the words
$\{u_1\dots u_n|\ n\in \N\}$ (resp. $\{v_n\dots v_1|\ n\in \N\}$). 
We also define $D_{vu}=\bigcup_{a\in A} D_{vu,a}$. Then
$D_{vu}$ is a geometric realization of the fixed point $vu\in A^{\Z}$, that is, 
the set of vertices of a broken line naturally generated by corresponding fundamental 
unit vectors $e_a\ (a\in A)$.

We project this broken line to make a self-similar tiling 
of the real line by tiles (intervals) corresponding to each letter.
This is done by associating intervals whose lengths are given 
by the entry of a left eigenvector $\ell=(\ell_a)_{a\in A}$. 
The corresponding expanding matrix
is of size $1$ and equal to the Perron Frobenius root of $M$.
Define
$\psi:D_{u,v}\rightarrow \R$ by
$
\psi(\sum_i^n e_{u_i})=\sum_{i=1}^n \ell_{u_i} 
$
and 
$
\psi(-\sum_i^n e_{v_i})=-\sum_{i=1}^n \ell_{v_i}
$
according to the domain $D_u$ or $D_v$.
Put $
\Lam_a=\{ \psi(v) |\ v\in D_{u,a} \}
$
for $a\in A$. We normalize the eigenvector $\ell$ so that $\psi$ becomes 
the orthogonal projection to the $1$-dimensional subspace $\pi^{-1}(0)$
generated by the expanding vector of $M$. Then this 
is exactly the set of left end points of intervals which consists 
the tiling.
It is clear that $\psi:D_{u,a}\rightarrow \Lam_a$ is bijective and  preserves 
addition structure, i.e., if $x\pm y\in D_{u,a}$ for $x,y\in D_{u,a}$ then
$\psi(x\pm y)=\psi(x)\pm \psi(y)$ holds in $\Lam_a$ and vice versa.
By this choice of the length, $\Lb=(\Lam_a)_{a\in A}$ forms a substitution 
multi-colored Delone set. 
When $s$ is a Pisot substitution, $\Lb$ is a substitution 
multi-colored Meyer set. The closure of the
set of translations $\{ \Lb-t|\ t\in \R\}$ by local topology forms a compact set $X$ 
and $(X,\R)$ is a topological dynamical system. 
Transferring the results on self-affine tilings (\cite{Solomyak:97}), 
by primitivity of $s$, this system is minimal and uniquely ergodic.
Moreover the system  
$(X,\R)$ is not weakly mixing if and only if $s$ is a Pisot substitution .
Clark and Sadun \cite{Clark-Sadun:03} showed
that if $s$ is an irreducible Pisot substitution, then 
$(X,\R)$ shows pure discrete spectrum if and only if $(\overline{S^{\Z}vu},S)$ 
does.
Therefore we can use techniques developed in the tiling dynamical system to our problem.

\subsection{Algebraic coincidence for $D_u$}
In this setting the algebraic coincidence in \cite{Lee:07} reads

\begin{equation}\label{AC}
\exists a\ \exists n\in N\  \exists \eta'\in \R\quad 
\beta^n \bigcup_{a\in A} (\Lam_a-\Lam_a)\in \Lam_a-\eta'
\end{equation}
and the projection $\psi$ is bijective, (\ref{AC}) is equivalent to

\begin{equation}\label{AC2}
\exists a\ \exists n\in N\ \exists \eta\in \R^A
\quad M^n(\cup_{a\in A} (D_{u,a}-D_{u,a}))\in D_{u,a}-\eta. 
\end{equation}
Clearly we see $\eta\in D_{u,a}$.
By primitivity of $s$, we easily see that for any $a,b\in A$, there exists 
$\ell\in \N$ such that
\begin{equation*}
    \beta^{\ell} (\Lam_b-\Lam_b)\subset (\Lam_a-\Lam_a)
\end{equation*}
or
\begin{equation}
\label{prm}
    M^{\ell} (D_{u,b}-D_{u,b})\subset (D_{u,a}-D_{u,a}).
\end{equation}
Yet we need another result depending heavily on irreducibility of substitution:

\begin{lemme}
\cite{sing, bk}
\label{height}
Let $s$ be a primitive irreducible substitution and $\Lb$ be 
an associated substitution Delone multi-color set in $\R$.
Then we have
\begin{eqnarray} 
\left\langle \bigcup_{a\in A}(\Lam_a - \Lam_a) \right\rangle = \left\langle (\bigcup_{a \in A} \Lam_a) - (\bigcup_{a \in A} 
\Lam_a) \right\rangle. 
\end{eqnarray}
\end{lemme}
Here $\langle X \rangle$ stands for the additive subgroup of $\R^A$
generated by the set described in $X$.
Since 
$$
\psi^{-1}\left(\left\langle (\bigcup_{a \in A} \Lam_a) - (\bigcup_{a \in A} 
\Lam_a) \right\rangle\right)
$$
contains all fundamental unit vector $e_a\ (a\in A)$, it clearly coincides with $\Z^d$. Therefore
Lemma \ref{height} implies
\begin{eqnarray}
\label{One}
\left\langle \bigcup_{a\in A}(D_{u,a} - D_{u,a}) \right\rangle = \Z^d
\end{eqnarray}

\subsection{Proof of the existence of an inner point}

Note that the substitution matrix $M$ is contained in $GL(d,\Z)$, 
because $s$ is a Pisot unit substitution.

Without loss of generality, we assume 
that $u$ begins with $a\in A$, which implies $0\in D_{u.a}$.  
We will prove that that 
there exists $N\in \N$ such that 
$\pi(\eta)$ is an inner point of $D_{u,a}$ 
where $\eta\in D_{u,a}$ appeared in (\ref{AC2}).

Let
\[
    \varphi:   \begin{array}{ccc}
                     \Part(\Z^A)    &\to&       \Part(\Z^A)  \\
                     S              &\mapsto&   M^n(S-S)
                \end{array} 
    \qquad
    \mathcal{D}:   \begin{array}{ccc}
                 \Part(\Z^A)    &\to&       \Part(\Z^A)  \\
                 S              &\mapsto&   S-S
            \end{array},
\]
where $n \in \N$ is such that
\[
    \forall b \in A,\ M^n(D_{u,b} - D_{u,b}) \subseteq D_{u,a} - \eta.
\]

\begin{lemme}
    For all $k \in \N_{\geq 1}$ and $b\in A$, we have
    \[
        \varphi^{k} D_{u,b} \subseteq D_{u,a} - \eta.
    \]
\end{lemme}

\begin{proof}
    Easy, by induction.
\end{proof}

Let $\B$ be the unit ball of $\Part$ centered at the origin.

\begin{lemme} \label{l_rd}
    Let $P \subseteq \Z^A$ such that $\psi(P)$ is relatively dense in $\R_+$ and $\pi(P)$ is bounded.
    Then, there exists $R > 0$ such that
    \[
        Q_{\B} \subseteq \bigcup_{x \in P} B(x, R)\cap \Z^{A},
    \]
    where $B(x,R)$ is the ball centered at $x$ of radius $R$ in $\R^{A}$.
\end{lemme}

\begin{proof}
    Let $M > 0$ such that $\forall x \in \R_+,\ d(x, \psi(P)) \leq M$.
    There exists $C_1 > 0$ and $C_2>0$ such that
    \[
        \forall (x,y) \in (\R^A)^2,\ d(x,y) \leq C_1 d(\psi(x),\psi(y)) + C_2 d(\pi(x), \pi(y))
    \]
    depending on the choice of the left eigenvector for the map $\psi$, and the choice of the linear projection $\pi$.
    We choose $R = C_1 M + C_2(\diam(\pi(P \cup \{0\})) + 1)$.
    Let $x \in Q_{\B}$, then we have $\psi(x) \in \R_+$, and for $y \in P$ such that $d(x,P) = d(x,y)$ we have
    \[
        d(x, y) \leq C_1 d(\psi(x), \psi(P)) + C_2 d(\pi(x), \pi(y)) \leq C_1 M + C_2(\diam(\pi(P \cup \{0\}))+ 1) \leq R.
    \]
    Therefore, we have $x \in \bigcup_{y \in P} B(y, R)\cap \Z^A$.
\end{proof}

\begin{lemme}
    There exists $N \in \N$ such that
    \[
        Q_{\B} \subseteq \mathcal{D}^N \left(\bigcup_{b \in A} D_{u,b} - D_{u,b} \right).
    \]
\end{lemme}

\begin{proof}
    The set $\psi(D_{u,a})$ is relatively dense in $\R_+$ and $\pi(D_{u,a})$ is bounded. So, by lemma~\ref{l_rd}, there exists a $R > 0$ such that
    \[
        Q_{\B} \subseteq \bigcup_{x \in D_{u,a}} B(x, R)\cap \Z^A.
    \]
    And we have, by~(\ref{One}),
    \[
        \bigcup_{N \in \N} \mathcal{D}^N \left(\bigcup_{b \in A} D_{u,b} - D_{u,b} \right)
         = \left\langle \bigcup_{b\in A}(D_{u,b} - D_{u,b}) \right\rangle = \Z^A.
    \]
    Thus there exists $N \in \N_{\geq 3}$ large enough such that
    \[
        M^{-\ell} B(0, R)\cap \Z^A \subseteq \mathcal{D}^{N-2} \left(\bigcup_{b \in A} D_{u,b} - D_{u,b} \right).
    \]
Multiplying by $M^{\ell}$ and using (\ref{prm}), we see
        \[
        B(0, R)\cap \Z^A \subseteq \mathcal{D}^{N-2} \left(D_{u,a} - D_{u,a}\right)=\mathcal{D}^{N-1}D_{u,a},
    \]
    Then, we have
    \[
        Q_{\B} \subseteq \bigcup_{x \in D_{u,a}} B(x, R) \cap \Z^A
                    \subseteq B(0, R)\cap \Z^A - \mathcal{D}^{N-1}(D_{u,a})
                    \subseteq \mathcal{D}^{N} D_{u,a}.
    \]
    Here we used the fact $x\in \mathcal{D}^{k} D_{u,a}$ for any $k\in \N_{\geq 1}$, because $0\in D_{u,a}$.
\end{proof}

Using these lemmas, we have the inclusion
    \[
        M^{nN} Q_{\B} \subseteq M^{nN} \mathcal{D}^N D_{u,a} = \varphi^{N} D_{u,a} \subseteq D_{u,a} - \eta.
    \]
And this implies that $D_{u,a}$ contains $M^{nN} Q_{\B} + \eta$, therefore it has non-empty interior.

\section{Computation of the interior} \label{sci}

In this section, we show that the interior of some subsets of $\Z^d$, for the topology defined in the subsection~\ref{topo}, can be described by a computable regular language.
This gives a way to decide the Pisot substitution conjecture for any given irreducible Pisot unit substitution. 

\subsection{Regular languages}

	Let $\Sigma$ be a finite set, and let $\Sigma^* = \bigcup_{n \in \N} \Sigma^n$ be the set of finite words over the alphabet $\Sigma$. A subset of $\Part(\Sigma^*)$, is called a \defi{language} over the alphabet $\Sigma$.
	We say that a language $L$ over an alphabet $\Sigma$ is \defi{regular} if the set
	\[
		\set{u^{-1}L}{u \in \Sigma^*}
	\]
	is finite, where $u^{-1}L := \set{v \in \Sigma^*}{uv \in L}$.
	
	An \defi{automaton} is a quintuplet $\A = (\Sigma, Q, I, F, T)$, where $\Sigma$ is a finite set called \defi{alphabet}, $Q$ is a finite set called \defi{states}, $I \subseteq Q$ is the set of \defi{initial states}, $F \subseteq Q$ is the set of \defi{final states}, and $T \subseteq Q \times \Sigma \times Q$ is the set of \defi{transitions}.
	We denote by $p \xrightarrow{ t } q$ a transition $(p,t,q) \in T$, and we will write
	\[
		q_0 \xrightarrow{ t_1 } q_1 \xrightarrow{ t_2 } ... \xrightarrow{ t_n } q_n \in T
	\]
	when for all $i = 1, 2, ..., n$ we have $(q_{i-1}, t_i, q_i) \in T$. 
	We call \defi{language recognized} by $\A$ the language $L_\A$ over the alphabet $\Sigma$ defined by
	\[
		L_\A = \set{ u \in \Sigma^* }{ \exists (q_i)_{i} \in Q^{\abs{u}+1},\ q_0 \in I, q_{\abs{u}} \in F, \text{ and } q_0 \xrightarrow{ u_1 } q_1 \xrightarrow{ u_2 } ... \xrightarrow{ u_{\abs{u}} } q_{\abs{u}} \in T  }.
	\]
	
	The following proposition is a classical result about regular languages \\ (see~\cite{KN,HU,Saka, carton}).
	\begin{prop}
		A language is regular if and only if it is the language recognized by some automaton.
	\end{prop}
	
	We say that an automaton is deterministic if $I$ has cardinality one, and if for every state $q \in Q$ and every letter $t \in \Sigma$, there exists at most one state $q' \in Q$ such that $(q, t, q') \in T$ is a transition.
	
	The \defi{minimal automaton} of a regular language $L$, is the unique deterministic automaton recognizing $L$ and having the minimal number of states. Such automaton exists, is unique, and the number of states is equal to the cardinal of the set $\set{u^{-1}L}{u \in \Sigma^*} \backslash \{ \emptyset \}$.
	To an automaton, we can associate the \defi{adjacency matrix} in $M_Q(\Z)$ whose $(s', s)$ coefficient is the number of transitions from state $s$ to state $s'$.
%
	We denote by $\trL$ the mirror of a language $L$.
	\[
		\trL := \set{ u_n u_{n-1} ... u_1 u_0 }{ u_0 u_1 ... u_{n-1} u_n \in L }.
	\]

\subsection{Discrete line associated to a regular language} 
	Given a word $u$  over an alphabet $\Sigma \subseteq \Z^d$, and a matrix $M \in M_d(\Z)$, we define
	\[
		Q_{u,M} = \sum_{k=0}^{\abs{u}-1} M^i u_i.
	\]
	Given a language $L$ over an alphabet $\Sigma \subseteq \Z^d$, and a matrix $M \in M_d(\Z)$, we define the following subset of $\Z^d$.
	\[
		Q_{L, M} = \set{ Q_{u,M} }{ u \in L } = \set{ \sum_{k=0}^{\abs{u}-1} M^i u_i }{ u \in L }.
	\]
	We will also call this set a discrete line, because when $M$ has a Pisot number as eigenvalue and no other eigenvalue of modulus greater than one, then this set stays at bounded distance of a line of $\R^A$ -- line which is the eigenspace of the matrix $M$ for the Pisot eigenvalue. And we show now that every discrete line coming from a substitution is also the discrete line of some regular language.
	When it will be clear from the context what is the matrix, we will simply write $Q_u$ and $Q_L$.
	
	\begin{rem}
		The notation $Q_S$ was also defined for a part $S \subseteq \Part$,
		but there is no ambiguity, because parts of $\Part$ and languages are always different objects, and we use the same notation because in both cases it represents a discrete line.
	\end{rem}
	
	To a substitution $s$ over the alphabet $A$, and $a,b \in A$, we associated the following deterministic automaton $\A^s_{a,b}$ with
	\begin{itemize}
		\item set of states $A$,
		\item initial state $a$,
		\item set of final states $\{b\}$,
		\item alphabet
		\[
		    \Sigma = \set{\Ab(u)}{u \in A^* \text{ such that } \exists (c,v) \in A \times A^*, \ s(c) = uv \text{ and } \abs{v} > 0},
		\]
		\item set of transitions
		\[
		    T = \set{(c,t,d) \in A \times \Sigma \times A}{ \exists u,v \in A^*,\ s(c) = u d v \text{ and } \Ab(u) = t}.
		\]
	\end{itemize}
	We denote by $L^s_{a,b}$ the language of this automaton.
	We denotes by $\A^s$ the automaton $\A^s_{a,b}$ where we forget the data of the initial state and the set of final states.
	
	\begin{rem}
		This automaton is the abelianisation of what we usually call the prefix automaton. 
	\end{rem}
	
	\begin{rem}
		For a substitution $s$ and two letters $a$ and $b$, the language $L^s_{a,b}$ has little to do with what we usually call the language of the substitution $s$ (i.e. the set of finite factors of periodic points of $s$).
		The alphabet of $L^s_{a,b}$ is not even the same as the alphabet of the substitution $s$.
	\end{rem}
	
	\begin{prop} \label{prop_DQ}
		If $u$ is a fixed point of a substitution $s$ whose first letter is $a$, then we have for every letter $b$,
		\[
			D_{u,b} = Q_{\trL^s_{a,b}, M_s}.
		\]
	\end{prop}
	
	\begin{rem}
		This proposition corresponds to write elements of the discrete line $D_{u,b}$ using the Dumont-Thomas numeration.
	\end{rem}
	
	\begin{rem}
		If we want to describe the left infinite part of the discrete line associated to a bi-infinite fixed point of the substitution $s$, we have to consider the automata $\A^{\transp{s}}_{a,b}$ where $\transp{s}$ is the reverse substitution of $s$ -- that is $\forall a \in A, \transp{s}(a) = \transp{\left(s(a) \right)}$. We can also describe a bi-infinite discrete line with only one automaton over the bigger alphabet $\Sigma_s \cup -\Sigma_{\transp s}$.
	\end{rem}
	
	\begin{rem}
		The automaton $\A^s$ permits to compute easily the map $E_1(s)$ defined in~\cite{AI}:
		\[
			E_1(s) (x, e_a) = \sum_{a \xrightarrow{t} b \in T} (Mx + t, e_b),
		\]
		where $T$ is the set of transitions of $\A^s$. 
		
		We can also compute easily the map $E_1^*(s)$ when $\det(M) \in \{-1, 1\}$:
		\[
			E_1^*(s) (x, e_b^*) = \sum_{a \xrightarrow{t} b \in T} (M^{-1}(x - t), e_a^*).
		\]
		And we have
		\[
			(y, e_b) \in E_1(s)(x, e_a) \quad
			\Longleftrightarrow \quad
			a \xrightarrow{y - Mx} b \in T \quad
			\Longleftrightarrow \quad
			(x, e_a^*) \in E_1^*(s)(y, e_b^*).
		\]
	\end{rem}
    
    \begin{rem} \label{rem_fig}
	    The figure~\ref{fig_prop_DQ} illustrates the equivalence between the choice of \hfill
	    \begin{itemize}
	    	\item a prefix of $s^n(a)$ followed by a letter $b$,
	    	\item a path of length $n$ in a ordered tree to a letter $b$,
	    	\item a path of length $n$ in the prefix automaton with initial state $a$ and final state $b$,
	    	\item a word of length $n$ in a regular language over the alphabet $\{0,1\}$,
	    \end{itemize}
	    for the substitution $s: a \mapsto ab, b \mapsto ac, c \mapsto a$. \\
	    For example, the prefix $abaca$ of $s^3(a)$ corresponds to the word $101$, and the prefix $a$ corresponds to the word $001$.
    \end{rem}
    
    \begin{figure}[h]
    	\centering
    	\caption{
    	           Several description of the same thing -- see remark~\ref{rem_fig}
    	       } \label{fig_prop_DQ}
    	\tohide
    	{
    	    \includegraphics[scale=.49]{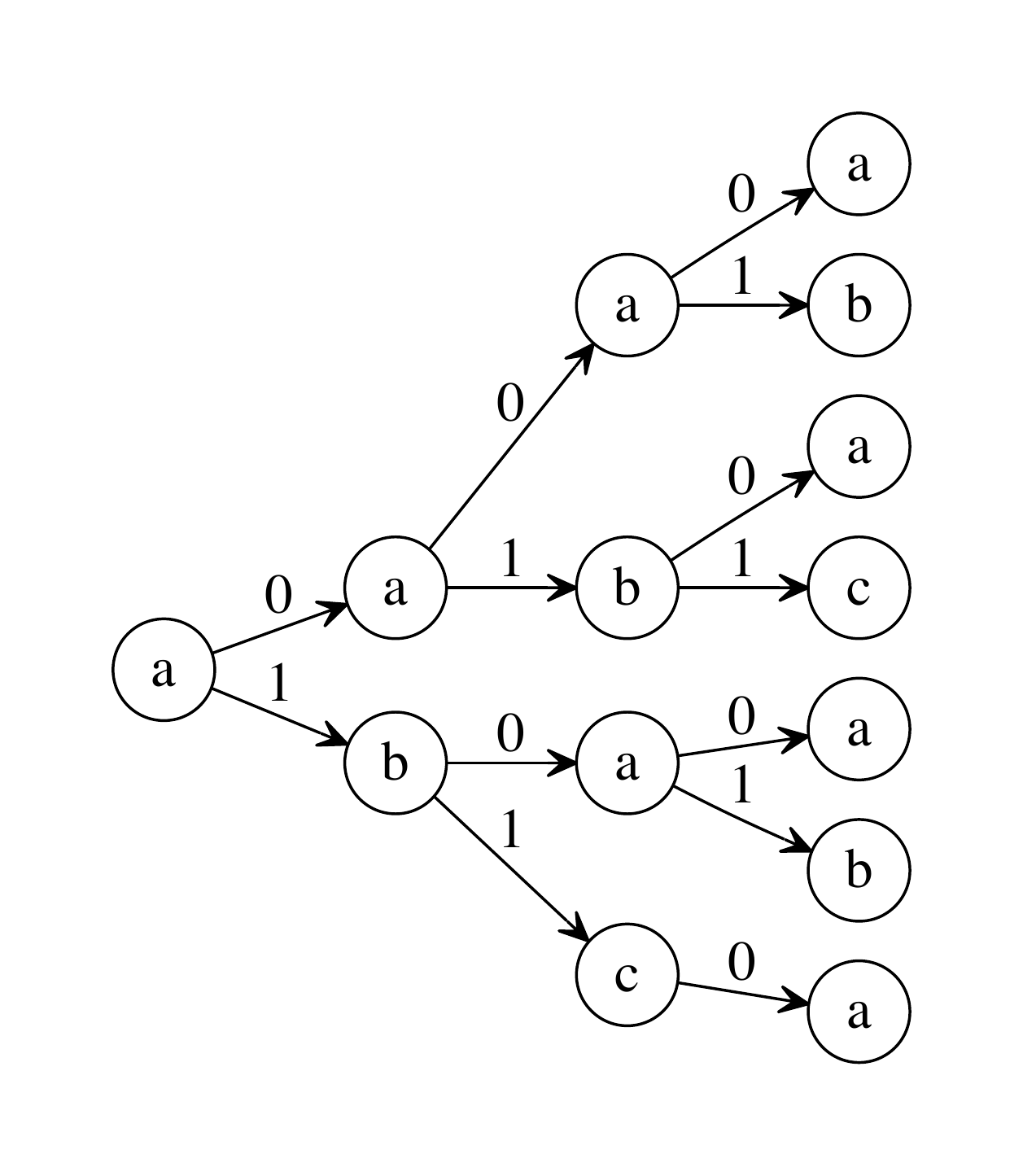} \qquad
    	    \includegraphics[scale=.49]{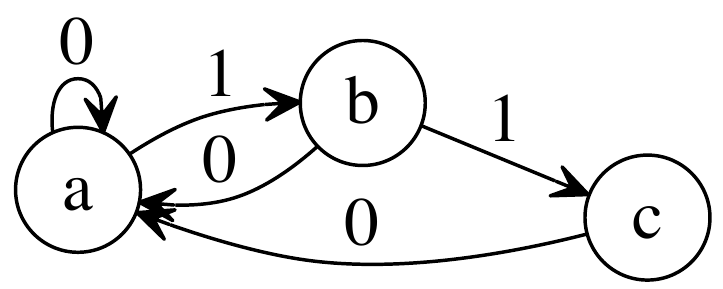}
    	}
    \end{figure}
	    
	\begin{proof}[Proof of the proposition~\ref{prop_DQ}]
	    The first idea is to describe the set of prefixes of $s^n(a)$ followed by a letter $b$, by words of length $n$ in the regular language $\trL^s_{a,b}$, like in the figure~\ref{fig_prop_DQ} but with a different alphabet.
	    
	    \begin{lemme} \label{lbijDQ}
	        For every $n \in \N$, there exists a natural map
	        \[
	            \varphi_n: \bigcup_{b \in A} \set{v \in A^*}{vb \text{ prefix of }s^n(a)} \to \bigcup_{b \in A} \set{w \in L^s_{a,b}}{\abs{w} = n}
	        \]
	        such that
	        \[
	            \forall b \in A,\ \varphi_n \left( \set{v \in A^*}{vb \text{ prefix of }s^n(a)} \right) = \set{w \in L^s_{a,b}}{\abs{w} = n}.
	        \]
	    \end{lemme}
	    
	    \begin{proof}
	        By induction on $n$. For $n = 0$, the map $\varphi$ is uniquely defined.
	        If $n \geq 1$, for $v \in A^*$ such that $vb$ is a prefix of $s^n(a)$, there exists a unique uplet of words $(v', v'', v''') \in (A^*)^3$, and an unique letter $c \in A$ such that $s(v'c) = vbv'''$, where $s(c) = v''bv'''$.
	        We define $\varphi_n(v) = \varphi_{n-1}(v') \Ab(v'')$.
	        This is a word of length $n$ in the regular language $L^s_{a,b}$, because by induction we have $\varphi_{n-1}(v') \in L^s_{a,c}$, with $\varphi_{n-1}(v')$ of length $n-1$, and the equality $s(c) = v''bv'''$ implies that there exists a transition for state $c$ to state $b$ labeled by $\Ab(v'')$ in the automaton $\A^s$.
	    \end{proof}
	    
	    We check that the formulae linking the abelianisation of the prefix and the corresponding word in $\trL^s_{a,b}$ is the one expected. 
	    
	    \begin{lemme}
	        For every $v \in A^*$ such that $vb \text{ is a prefix of }s^n(a)$, we have
	        \[
	            \Ab(v) = Q_{\transp{\varphi_n(v)}},
	        \]
	        where $\varphi_n$ is the map defined by lemma~\ref{lbijDQ} above.
	    \end{lemme}
	    
	    \begin{proof}
	        For every such word $v$, the map $\varphi_n$ gives a unique sequence $(v_k, c_k, w_k) \in A^* \times A \times A^*$ such that we have $\forall\ 0 \leq k \leq n-1$, $s(c_k) = v_{k+1} c_{k+1} w_{k+1}$, with $c_0 = a$, $v_0 = w_0 = \epsilon$, and $c_n = b$.
	        And, we have
	        \[
	            s^n(a) = vbw = s^{n-1}(v_1)s^{n-2}(v_2) ... s(v_{n-1})v_n c_n w_n s(w_n) ... s^{n-2}(w_2) s^{n-1}(w_1).
	        \]
	        Therefore we have
	        \begin{eqnarray*}
	            \Ab(v)  &=& \Ab\left(s^{n-1}(v_1)s^{n-2}(v_2) ... s(v_{n-1})v_n \right) \\
	                    &=& M^{n-1} \Ab(v_1) + M^{n-2} \Ab(v_2) + ... + M \Ab(v_{n-1}) + \Ab(v_n) \\
	                    &=& Q_{\Ab(v_n) \Ab(v_{n-1}) ... \Ab(v_2) \Ab(v_1)} \\
	                    &=& Q_{\transp{\varphi_n(v)}}.
	        \end{eqnarray*}

	    \end{proof}
	    
	    With these two lemmas, we get
	    \begin{eqnarray*}
	        D_{u,b} &=& \set{\Ab(v)}{ vb \text{ prefix of } u } \\
	                &=& \bigcup_{n \in \N} \set{\Ab(v)}{ vb \text{ prefix of } s^n(a) } \\
	                &=& \bigcup_{n \in \N} \set{Q_{\transp{w}}}{ w \in L^s_{a,b}, \ \abs{w} = n } \\
	                &=& Q_{\trL^s_{a,b}}.
	    \end{eqnarray*}
	    This ends the proof of the proposition~\ref{prop_DQ}.

	\end{proof}

\subsection{Computation of the interior}

We have seen in the previous section that the subshift associated to an irreducible Pisot substitution has pure discrete spectrum as soon as the interior of a piece of the discrete line is non-empty (see theorem~\ref{thm_int}), for the topology defined in subsection~\ref{topo}.
In this section, we give a way to compute the interior (and hence to test the Pisot substitution conjecture) with the following

\begin{thm} \label{cint}
	Let $L$ be a regular language over an alphabet $\Sigma \subseteq \Z^A$, $M$ be an irreducible Pisot unimodular matrix, and $\pi$ be the projection on a hyperplane $\Part$ along the eigenspace of $M$ for its maximal eigenvalue $\beta$.
	Then, the language
	\[
	    \overset{\circ}{L} := \set{u \in L}{Q_u \in \overset{\circ}{Q}_L}
	\]
	is regular and verify 
	$Q_{\overset{\circ}{L}} = \overset{\circ}{Q}_L$.
	Moreover, this language $\overset{\circ}{L}$ is computable from $L$.
\end{thm}

\begin{rem}
	The language $\overset{\circ}{L}$ doesn't depend on the choice of the hyperplane $\Part$. 
\end{rem}

With this theorem, the criterion given by the theorem~\ref{thm_int} gives the following result:

\begin{cor}
	Let $s$ be an irreducible Pisot unit substitution over an alphabet $A$, and let $u \in A^\N$ be a fixed point whose first letter is $a \in A$.
	Then we have the equivalence between
	\begin{itemize}
	    \item There exists a letter $b \in A$ such that the regular language $\transp {\overset{\circ}{L^s}}_{a,b}$ is non-empty,
	    \item The subshift $(\overline{S^\N u}, S)$ has pure discrete spectrum.
	\end{itemize}
\end{cor}

Hence, this gives a way to test if one given irreducible Pisot unit substitution satisfy the Pisot substitution conjecture or not.
And, the Pisot substitution conjecture is equivalent to 
\begin{conj} \label{conj_L}
	For any irreducible Pisot unit substitution $s$ over an alphabet $A$ and for any letters $a,b \in A$, the regular language $\transp {\overset{\circ}{L^s}}_{a,b}$ is non-empty.
\end{conj}


\subsection{Proof of the theorem~\ref{cint}}

In order to compute the interior, we need a big enough alphabet.
	
\begin{lemme} \label{alpha}
	For any Pisot unit primitive matrix $M \in M_d(\N)$, there exists $\Sigma' \subseteq \Z^A$ such that $0 \in \overset{\circ}{Q}_{{\Sigma'}^*, M}$.
\end{lemme}

\begin{proof}
	Let's consider any substitution $s$ whose incidence matrix is the irreducible unit Pisot matrix $M$.
	Let $u$ be a periodic point for this substitution.
	We know that $\pi(D_u)$ is bounded and is a fundamental domain for the action of the lattice $\pi(\Gamma_0)$ on $\pi(\Z^A)$, where $\Gamma_0$ is the subgroup of $\Z^A$ spanned by $e_a - e_b$, $a,b \in A$.
	Hence, there exists a finite subset $S \subseteq \Gamma_0$ such that $D_u + S = \set{x + y}{(x,y) \in D_u \times S}$ contains zero in its interior.
	Then, the alphabet $\Sigma' = \Sigma_s + S$ satisfy that $0 \in \overset{\circ}{Q}_{\Sigma' \Sigma_s^*} \subseteq \overset{\circ}{Q}_{\Sigma'^*}$.	
\end{proof}

The alphabet given by this lemma is not optimal.
Here are two conjectures that gives natural choices of alphabet.
The first one gives an alphabet of minimal size, and the second one gives the alphabet $\Sigma$ that naturally comes from the substitution.

\begin{conj}
	For all irreducible unit Pisot matrix $M$ with spectral radius $\beta$, we have $0 \in \overset{\circ}{Q}_{{\Sigma'}^*}$ , for $\Sigma' = \{ -1, 0, 1, 2, ..., \ceil{\beta}-2\}$.
\end{conj}

\begin{conj}
	For all irreducible unit Pisot substitution $s$, we have $\overset{\circ}{Q}_{{\Sigma_s}^*} \neq \emptyset$.
\end{conj}

\begin{rem}
	This last conjecture is a consequence of the Pisot substitution conjecture. But it should be easier to solve.
\end{rem}

\begin{rem}
	We cannot assume in this last conjectures that the interior always contains $0$, since we can only get the positive part of the hyperplane $\Part$ with Pisot numbers whose conjugates are positive reals numbers.
	Nevertheless, if we have only $\overset{\circ}{Q}_{{\Sigma'}^*} \neq \emptyset$, then the set $L_{int}$ computed in the proof of the theorem~\ref{cint} satisfy
	\[
		\overset{\circ}{Q_L} \subseteq Q_{L_{int}} \subseteq \overline{\overset{\circ}{Q_L}},
	\]
	so we have $\overset{\circ}{Q_L} = \emptyset \Longleftrightarrow L_{int} = \emptyset$. Hence we can decide if $Q_{L}$ has empty interior or not by computing $L_{int}$ with this alphabet $\Sigma'$.
\end{rem}

The following theorem is also useful to compute the interior. It is a variant of the main theorem of~\cite{me}.

\begin{thm} \label{thm_rel}
	Consider two alphabets $\Sigma$ and $\Sigma'$ in $\Z^A$, and a matrix $M \in M_A(\Z)$ without eigenvalue of modulus one.
	Then the language
	\[
		\Lrel := \set{(u,v) \in (\Sigma' \times \Sigma)^*}{ Q_u = Q_v }
	\]
	is regular.
\end{thm}

\begin{rem}
	This language $\Lrel$ is related to what is usually called the zero automaton.
	See~\cite{fp} and~\cite{fs} for more details.
\end{rem}

\begin{proof}[proof of the theorem~\ref{cint}]
	Consider the language
	\[
		L_{int} := Z(S(Z(p_1(\Sigma'^* \times L0^* \cap \Lrel)))),
	\]
	where
	\begin{itemize}
		\item $\Sigma'$ is an alphabet given by the lemma~\ref{alpha} and containing $0$,
		\item $p_1 : (\Sigma' \times \Sigma)^* \to \Sigma'^*$ is the word morphism such that
		\[
		    \forall (x,y) \in \Sigma' \times \Sigma,\ p_1((x,y)) = x,
		\]
		\item $\Lrel$ is the language defined in theorem~\ref{thm_rel},
		\item for any language $L$ over the alphabet $\Sigma'$,
		\[
		    S(L) := \set{u \in \Sigma'^*}{ u\Sigma'^* \subseteq L },
		\]
		\item for any language $L$ over the alphabet $\Sigma'$,
		\[
		    Z(L) := \set{u \in \Sigma'^*}{ \exists n \in \N,\ u0^n \in L }.
		\]
	\end{itemize}
	Then, we have
	\[
		L_{int} = \set{ u \in \Sigma'^* }{ Q_u \in Q_{L_{int}} }.
	\]
	
	Indeed, for all $u \in \Sigma'^*$ we have
	\begin{eqnarray*}
		u \in L_{int}	& \Longleftrightarrow & \exists n \in \N,\ u0^n \in S(Z(p_1(\Sigma'^* \times L0^* \cap \Lrel))), \\
					& \Longleftrightarrow & \exists n \in \N,\ u0^n\Sigma'^* \subseteq Z(p_1(\Sigma'^* \times L0^* \cap \Lrel)), \\
					& \Longleftrightarrow & \exists n \in \N,\ \forall v \in \Sigma'^*,\ \exists k \in \N,\ u0^nv0^k \in p_1(\Sigma'^* \times L0^* \cap \Lrel), \\
					& \Longleftrightarrow & \exists n \in \N,\ \forall v \in \Sigma'^*,\ \exists k \in \N,\ \exists w \in L0^*,\ (u0^nv0^k, w) \in \Lrel, \\
					& \Longleftrightarrow & \exists n \in \N,\ \forall v \in \Sigma'^*,\ Q_{u0^nv} \in Q_{L}, \\
					& \Longleftrightarrow & \exists n \in \N,\ Q_{u} + M^{n+\abs{u}}Q_{\Sigma'^*} \subseteq Q_{L}, \\
					& \Longleftrightarrow & Q_{u} \in \overset{\circ}{Q}_{L}.
	\end{eqnarray*}
	
	But we can assume that $\Sigma \subseteq \Sigma'$ up to replace $\Sigma'$ by $\Sigma \cup \Sigma'$.
	Then, we get the language $\overset{\circ}{L}$ by taking the intersection with $L$:
	\[
		\overset{\circ}{L} = L \cap L_{int}.
	\]
\end{proof}

\begin{rem}
	If we just want to test the non-emptiness of the language $\overset{\circ}{L}$, it is not necessary to compute all what is done in this proof.
	For example, the computation of the language $L_{int}$ is enough (and we do not need that $\Sigma \subseteq \Sigma'$). And we don't even need to compute completely $L_{int}$ if we only want to test if it is non-empty. And it is enough to have $\Sigma'$ such that $Q_{\Sigma'^*}$ has non-empty interior.
\end{rem}

\subsection{Examples}

\begin{ex}
	For the Fibonnacci and for the Tribonnacci substitutions, we get $\transp {\overset{\circ}{L^s}_{a, b}} = \transp L^s_{a, b}$, for $a$ the first letter of the fixed point $u$, and any letter $b$. Therefore the sets $D_{u,b}$ are open: $\overset{\circ}{D}_{u,b} = D_{u,b}$ (and we can check that they are also closed).
\end{ex}

\begin{ex} \label{ex_tri_flip}
	For the "flipped" Tribonnacci substitution:
	\[
		\begin{array}{l}
			a \mapsto ab\\
			b \mapsto ca\\
			c \mapsto a
		\end{array}
	\]
	the minimal automaton of the language $\overset{\circ}{\transp L_{a,a}^s}$ has $79$ states ($80$ states for $\overset{\circ}{\transp L_{a,b}^s}$, $81$ for $\overset{\circ}{\transp L_{a,c}^s}$). This automaton is plotted in figure~\ref{fig_tri_flip_int}, and the sets $\pi(D_{u,a})$ and $\pi(\overset{\circ}{D}_{u,a})$ for the fixed point $u$ are drawn in figure~\ref{fig_tri_flip}.
\end{ex}

\begin{figure}[h]
	\centering
	\caption{Minimal automaton of the language $\overset{\circ}{\transp L_{a,a}^s}$ of the example~\ref{ex_tri_flip}. The labels $0$ correspond to the null vector, the labels $1$ correspond to the vector $e_a$, and the labels $b^2 - b - 1$ correspond to the vector $e_c$. Final states are the double circles, and the initial state is the bold circle.} \label{fig_tri_flip_int}
	\tohide
	{
	    \includegraphics[scale=.49]{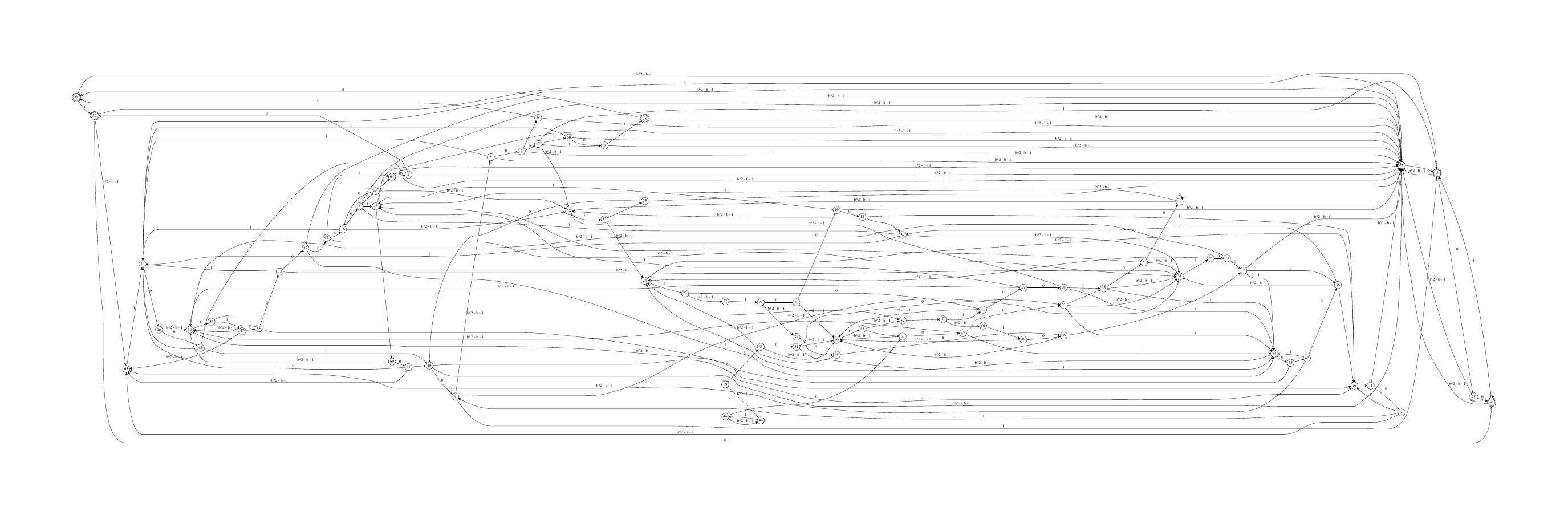}
	}
\end{figure}


\begin{figure}[h]
	\centering
	\caption{The sets $\pi(D_{u, a})$ (in gray and blue) and $\pi(\overset{\circ}{D}_{u,a})$ (in blue) for the example~\ref{ex_tri_flip}} \label{fig_tri_flip}
	\tohide{
	    \includegraphics[scale=.3]{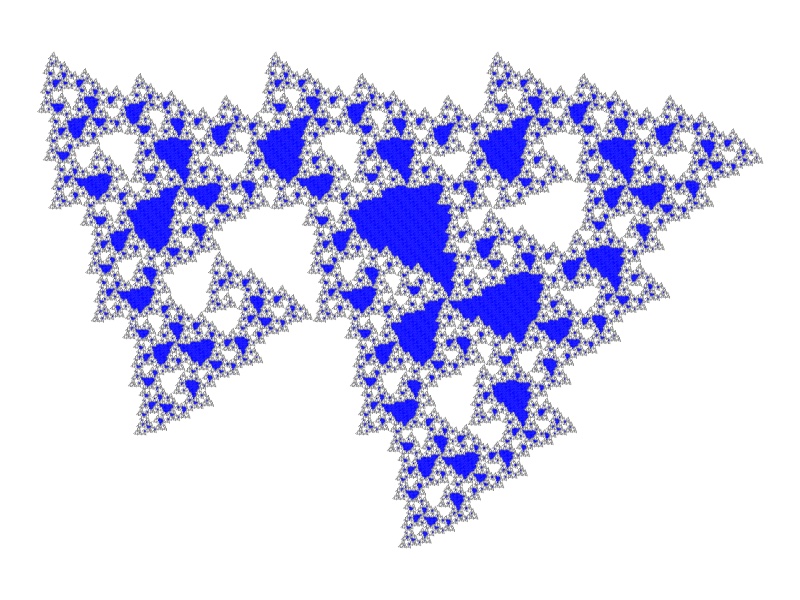}
	}
\end{figure}

\begin{ex} \label{ex_Pisot_min}
	For the following substitution associated to the smallest Pisot number:
	\[
		\begin{array}{l}
			a \mapsto b\\
			b \mapsto c\\
			c \mapsto ab
		\end{array}
	\]
	the minimal automaton of the language $\overset{\circ}{\transp L_{a,a}^s}$ has $1578$ states ($1576$ states for $\overset{\circ}{\transp L_{a,b}^s}$, $1577$ for $\overset{\circ}{\transp L_{a,c}^s}$). The sets $\pi(D_{u,a})$ and $\pi(\overset{\circ}{D}_{u,a})$ are plotted on figure~\ref{fig_Pisot_min}, where $u$ is the periodic point starting by letter $a$.
\end{ex}


\begin{figure}[h]
	\centering
	\caption{The sets $\pi(D_{u, a})$ (in gray and blue) and $\pi(\overset{\circ}{D}_{u,a})$ (in blue) for the example~\ref{ex_Pisot_min}. Whole set at the left, and a zoom on it at the right.} \label{fig_Pisot_min} 
    \tohide{
	    \includegraphics[scale=.3]{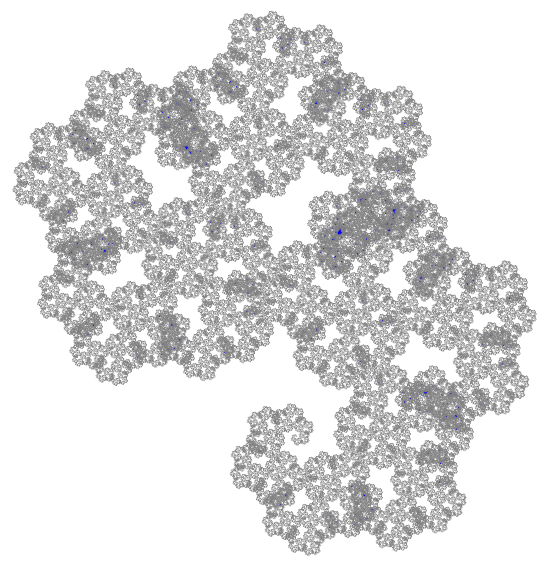}
	    \includegraphics[scale=.25]{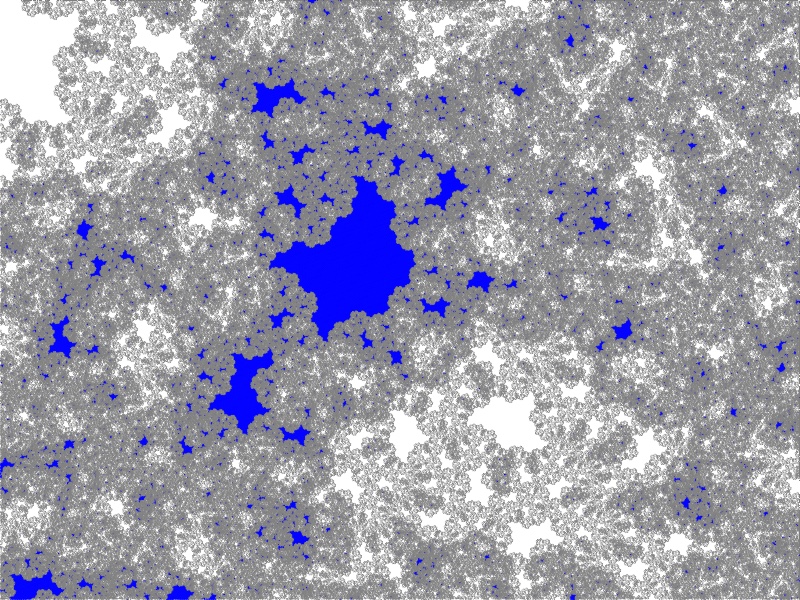}
	}
\end{figure}

\begin{rem}
	The first author have implemented the computation of the interior in the Sage mathematical software.
	The above examples has been computed using this implementation which is partially available here : \url{https://trac.sagemath.org/ticket/21072}.
	Unfortunately these tools are not easy to install and not well documented for the moment.
\end{rem}

\begin{rem}
	To prove the Pisot substitution conjecture, it is enough for each irreducible Pisot substitution $s$ and for any letter $a$, to find one particular "canonical" word in the language $\overset{\circ}{\transp L_{a,a}^s}$ in order to prove it is non-empty.
\end{rem}

\section{Pure discreteness for various infinite family of substitutions}

\subsection{Proof of pure discreteness using a geometrical argument}

Using the theorem~\ref{thm_int}, we can prove the Pisot substitution  conjecture for a new infinite family of substitutions:

\begin{thm}
	Let $k \in \N$, and let
	\[
		s_k : \left\{ \begin{array}{l}
				a \mapsto a^kbc\\
				b \mapsto c\\
				c \mapsto a
			\end{array} \right.
	\]
	where $a^k$ means that the letter $a$ is repeated $k$ times. 
	The subshift generated by the substitution $s_k$ is measurably conjugate to a translation on the torus $\T^2$.
\end{thm}

\begin{figure}[h]
	\centering
	\caption{Rauzy fractal of $s_{20}$} 
	\tohide{
	    \includegraphics[scale=.35]{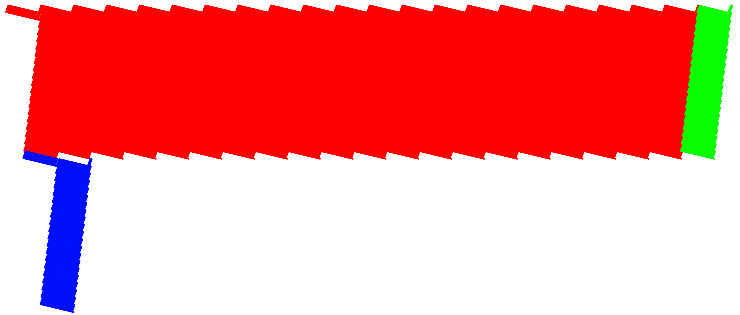}
	}
\end{figure}

\begin{proof}
	The strategy of the proof is to use the theorem~\ref{thm_int}.
	For $k \geq 1$, let $u$ be the fixed point of $s_k$ starting with letter $a$.
	We show that $\pi(\overset{\circ}{D}_{u,a}) \neq \emptyset$ by showing that the point
	\[
		t_k := \frac{k}{2} - \frac{\sqrt{k}}{2}I
	\]
	is not in the closure of $\pi(\Z^A \backslash D_a)$:
	\[
		t_k \not \in \bigcup_{l \in A \backslash \{ a \} } \overline{\pi(D_{u,l})} \cup \bigcup_{t \in \Gamma_0 \backslash \{ 0 \}} \overline{\pi(D_u + t)},
	\]
	where $A = \{a, b, c\}$ is the alphabet of the substitution $s_k$, $\Gamma_0$ is the group generated by $(e_i - e_j)_{i, j \in A}$, where $(e_i)_{i \in A}$ is the canonical basis of $\R^A$, $I$ denotes a complex number such that $I^2 = -1$, and $\pi$ is the projection along the eigenspace for the maximal eigenvalue of the incidence matrix
		\[
			M = \left(\begin{array}{rrr}
				k & 0 & 1 \\
				1 & 0 & 0 \\
				1 & 1 & 0
				\end{array}\right)
		\]
	of the substitution $s_k$, such that $\pi(e_a) = 1$, $\pi(e_b) = -\beta^2 + (k+1)\beta - (k-1)$ and $\pi(e_c) = \beta^2 - k\beta - 1$,
	where $\beta$ is the complex eigenvalue of $M$ such that $\Im(\beta) < 0$.
    Note that the characteristic polynomial of $M$ is $\chi_M(X) = X^3-k X^2 - X - 1$.
	
	In order to do that, we approximate the sets $D_{u,l}$ by union of balls :
	
	\begin{figure}[h]
		\centering
		\caption{Strategy to prove that $t_k \not\in \overline{\pi(\Z^d \backslash D_{u,a})}$ \\ Approximation of the sets $\pi(D_{u,l})$ and their translated copies, by disks, for $k=20$} 
		\tohide{
	        \includegraphics[scale=.18]{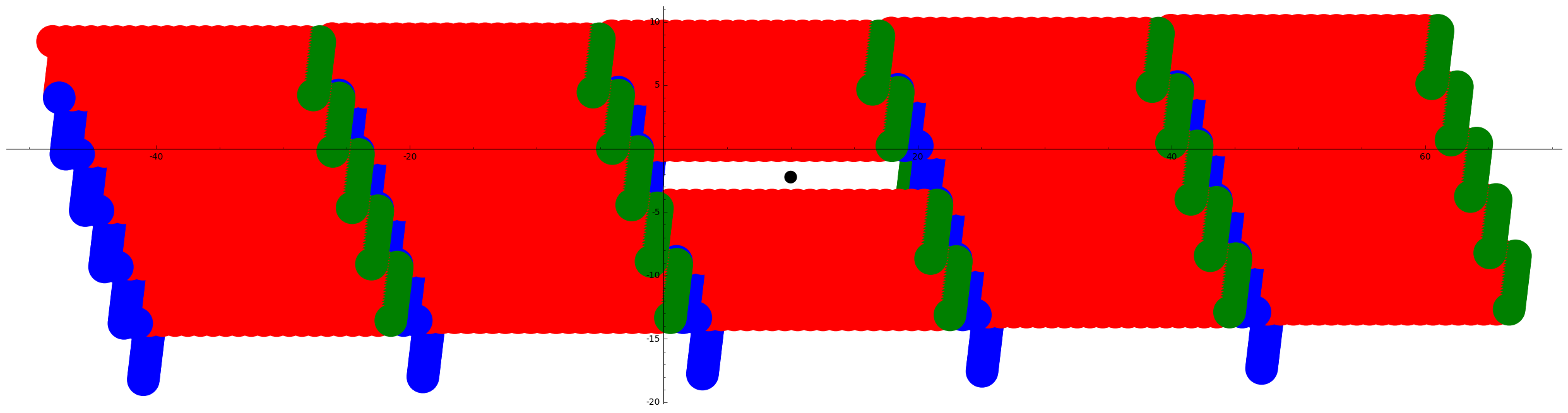}
	    }
	\end{figure}
	
	\begin{lemme}
		For all $k \geq 3$ and for every $l \in A$, we have the inclusion
		\[
			\overline{\pi(D_{u,l})} \subseteq \bigcup_{t \in S_{l}} B(t, \frac{1}{1 - \frac{1}{\sqrt{k}}}),
		\]
		where
		\begin{eqnarray*}
			S_{a} &=& \{ \gamma \beta \} \cup \set{ i + \beta j }{ (i,j) \in \{0, 1, ..., k-1\}^2 }, \\
			S_{b} &=& \set{ k + \beta i }{ i \in \{0, 1, ..., k-1\} }, \\
			S_{c} &=& \{ k \beta \} \cup \set{ \gamma + \beta i }{ i \in \{0, 1, ..., k-1\} },
		\end{eqnarray*}
		where $\gamma = - \beta^2 + (k+1) \beta + 1 = \beta - \frac{1}{\beta}$.
	\end{lemme}
	
	\begin{figure}[h]
			\centering
			\caption{Automaton describing $\pi(D_u)$} \label{fig_aut_sk} 
			\tohide{
    	        \begin{tikzpicture}[->,>=stealth',shorten >=1pt,auto,node distance=2.8cm, semithick]
    				  \tikzstyle{every state}=[fill=white, draw=black, minimum size=.5cm]
    				
    				  \node[initial,state,accepting] (A)                    {$a$};
    				  \node[state,accepting]         (B) [below right of=A] {$b$};
    				  \node[state,accepting]         (C) [above right of=B] {$c$};
    				
    				  \path (A) edge	[loop above]	node {$0$, $1$, ..., $k-1$} (A)
    				            	edge				node [below left] {$k$} (B)
    				            	edge	[bend left]		node {$\gamma$} (C)
    				        (B) edge 				node [below right] {$0$} (C)
    				        (C) edge [bend left]		node {$0$} (A);
    			\end{tikzpicture}
			}
		\end{figure}
	
	\begin{proof}
		By the proposition~\ref{prop_DQ}, for every $l \in A$, we have the equality
		\[
			\pi(D_{u,l}) = \set{\sum_{k=0}^{\abs{u}-1} u_i \beta^i }{ u \in \transp{L}_l }
		\]
		where $L_l$ is the language of the automaton of Figure~\ref{fig_aut_sk} where we replace the set of final states by $\{ l \}$. Indeed, this automaton is the automaton $\A_s$ of the substitution where we apply $\pi$ to labels of transitions.
		
		We get the proof of the lemma by considering words of length two, and by the inequality
		\[
			\abs{\sum_{k=2}^{\abs{u}-1} u_i \beta^i} \leq \sum_{k=2}^{\abs{u}-1} \max \set{\abs{t}}{t \in \Sigma} \abs{\beta}^i
											\leq \frac{1}{1 - \frac{1}{\sqrt{k}}}
		\]
		for any word $u$ over the alphabet $\Sigma$, where $\Sigma = \{0,1,..., k-1, k, \gamma\}$ is the alphabet of the languages $L_l$.
		Indeed, we have $\max \set{\abs{t}}{t \in \Sigma} = k$ and $\abs{\beta} \leq \frac{1}{\sqrt{k}}$ for $k \geq 3$.
	\end{proof}
	
	\begin{lemme}
		For every $k \geq 1$, we have the inequalities
		\[
			\frac{1}{\sqrt{k + \frac{2}{k}}} < \abs{\beta} < \frac{1}{\sqrt{k}}
		\]
		\[
			\sqrt{k} - \frac{1}{\sqrt{k}} < \abs{\gamma} < \sqrt{k + \frac{2}{k}} + \frac{1}{\sqrt{k}}
		\]
		\[
			-\frac{1}{k} < \Re(\beta) < -\frac{1}{k + \frac{2}{k}}
		\]
		\[
			-\frac{1}{\sqrt{k}} < \Im(\beta) < -\frac{1}{\sqrt{k + \frac{2}{k}}} + \frac{1}{k}
		\]
		where $\Re(\beta)$ is the real part of $\beta$, and $\Im(\beta)$ is the imaginary part.
	\end{lemme}
	\begin{proof}
		Let $\beta_+$ be the real conjugate of $\beta$.
		We have $\beta_+ = k + \frac{1}{\beta_+} + \frac{1}{\beta_+^2} > 0$, so
		\[
			k \leq \beta_+ \leq k + \frac{2}{k}.
		\]
		And we have $\abs{\beta}^2 = \frac{1}{\beta_+}$, hence we get the wanted inequalities for $\abs{\beta}$. 
		The inequalities for $\gamma = \beta - \frac{1}{\beta}$ follow.
		To get the real part, remarks that we have $k = \beta_+ + \beta + \overline{\beta} = \beta_+ + 2 \Re(\beta)$, and
		this gives $\Re(\beta) = -\frac{1}{2\beta_+} - \frac{1}{2 \beta_+^2}$.
		The inequalities for the imaginary part follow. 
	\end{proof}
	
	\begin{lemme}
		For all $k \geq 14$, we have $t_k \not\in \overline{\pi(D_{u,b})}$
	\end{lemme}
	
	\begin{proof}
		For all $i \in \{0, 1, ..., k-1\}$, we have $\abs{k + i \beta - t_k} = \abs{\frac{k}{2} + i \beta + \frac{\sqrt{k}}{2} I} \geq \frac{k}{2} - \abs{i \beta} - \frac{\sqrt{k}}{2} \geq \frac{k}{2} - \frac{3 \sqrt{k}}{2}$.
		This is greater than $\frac{1}{1 - \frac{1}{\sqrt{k}}}$ for $k \geq 14$.
	\end{proof}
	
	\begin{lemme}
		For all $k \geq 31$, we have $t_k \not\in \overline{\pi(D_{u,c})}$
	\end{lemme}
	
	\begin{proof}
		We have $\abs{\gamma + i \beta - t_k} = \abs{-\frac{k}{2} + \gamma + i \beta + \frac{\sqrt{k}}{2} I} \geq \frac{k}{2} - \abs{i \beta} - \abs{\gamma} - \frac{\sqrt{k}}{2} \geq \frac{k}{2} - \frac{3 \sqrt{k}}{2} - \sqrt{k + \frac{2}{k}} - \frac{1}{\sqrt{k}}$.
		This is greater than $\frac{1}{1 - \frac{1}{\sqrt{k}}}$ for $k \geq 31$.
		
		We have $\abs{k\beta - t_k} \geq \frac{k}{2} - \frac{3\sqrt{k}}{2}$.
		This is greater than $\frac{1}{1 - \frac{1}{\sqrt{k}}}$ for $k \geq 14$.
	\end{proof}
	
	Let us show now that the point $t_k$ is not in the translated copies of $\overline{\pi(D_u)}$ by the group $\pi(\Gamma_0)$.
	The group $\pi(\Gamma_0)$ is
	\[
		\pi(\Gamma_0) = \set{ c (\beta - k - 2) + d(\beta^2 - k\beta - 2) }{(c,d) \in \Z^2}.
	\]
	
	Let $t_{c,d} := c (\beta - k - 2) + d(\beta^2 - k\beta - 2)$.
	
	\begin{lemme} \label{la}
		For all $k \geq 8$ and for all $(c,d) \in \Z^2$ such that $\abs{c} \geq 1$ and $2\abs{c} \geq \abs{d}$, we have
		\[
			\abs{t_{c,d}} \geq k - 3 \sqrt{k}.
		\]
	\end{lemme}
	
	\begin{proof}
		We have
		\[
			\abs{t_{c,d}}	\geq \abs{c} (k+2) - \frac{\abs{c}}{\sqrt{k}} - \abs{d}(\frac{1}{k} + \sqrt{k} + 2)
						\geq \abs{c} (k - \frac{1}{\sqrt{k}} - \frac{2}{k} - 2\sqrt{k} - 2).
		\]
		This is greater than $k - 3 \sqrt{k}$ for $k \geq 8$.
	\end{proof}
	
	\begin{lemme} \label{li}
		For all $k \geq 9$ and for all $(c,d) \in \Z^2$ such that $\abs{c} \leq \abs{d}$ and $2 \leq \abs{d}$, we have
		\[
			\abs{\Im(t_{c,d})} \geq 2 \sqrt{k} - 3.
		\]
		And if moreover $\abs{d} \geq 3$, then we have $\abs{\Im(t_{c,d})} \geq 3 \sqrt{k} - 5$.
	\end{lemme}
	
	\begin{proof}
		We have $\abs{\Im(\beta^2-k\beta-2)} \geq k \abs{\Im(\beta)} - \abs{\beta}^2 \geq \frac{k}{\sqrt{k + \frac{2}{k}}} - 1 - \frac{1}{k}$.
		And we have $\abs{\Im(\beta - k -2)} = \abs{\Im(\beta)} \leq \frac{1}{\sqrt{k}}$.
		Hence,
		$\abs{\Im(t_{c,d})}	\geq \abs{d} \left( \frac{k}{\sqrt{k + \frac{2}{k}}} - 1 - \frac{1}{k} \right) - \abs{c} \frac{1}{\sqrt{k}}
						\geq 2 \left( \frac{k}{\sqrt{k + \frac{2}{k}}} - 1 - \frac{1}{k} - \frac{1}{\sqrt{k}} \right)$.
		This is greater than $2\sqrt{k} - 3$ for $k \geq 9$.
		If moreover $\abs{d} \geq 3$, then we have
		$\abs{\Im(t_{c,d})} \geq 3 \left( \frac{k}{\sqrt{k + \frac{2}{k}}} - 1 - \frac{1}{k} - \frac{1}{\sqrt{k}} \right)$,
		and this is greater than $3 \sqrt{k} - 5$ for $k \geq 6$.
	\end{proof}
	
	\begin{figure}[h]
		\centering
		\caption{Zone covered by the lemma~\ref{la} (in green), and by the lemma~\ref{li} (in blue), and points remaining (the red points are remaining only for $l = c$)}
		\tohide{
    	    \begin{tikzpicture}[scale = .45]
    			\begin{scope}
    				\clip(-4.5, -4.5) rectangle (4.5, 4.5);
    				\draw[step=1.0,gray,thin] (-4.5,-4.5) grid (4.5,4.5);
    				\draw[fill=blue!25, fill opacity=.5] (-4.5, 4.5) -- (-2,2) -- (2,2) -- (4.5, 4.5);
    				\draw[fill=blue!25, fill opacity=.5] (-4.5, -4.5) -- (-2,-2) -- (2,-2) -- (4.5, -4.5);
    				\draw[fill=blue!25, fill opacity=.5] (-4.5, 4.5) -- (-3,3) -- (3,3) -- (4.5, 4.5);
    				\draw[fill=blue!25, fill opacity=.5] (-4.5, -4.5) -- (-3,-3) -- (3,-3) -- (4.5, -4.5);
    				\draw[fill=green!25, fill opacity=.5] (-4.5, -9) -- (-1,-2) -- (-1,2) -- (-4.5, 9);
    				\draw[fill=green!25, fill opacity=.5] (4.5, -9) -- (1,-2) -- (1,2) -- (4.5, 9);
    			\end{scope}
    			\draw[->, thick] (-4.5,0)--(4,0) node[right]{$c$};
    			\draw[->, thick] (0,-4.5)--(0,4) node[above]{$d$};
    			\fill (0, 1)  circle[radius=.1];
    			\fill (0, -1)  circle[radius=.1];
    			\fill[color=red] (0, 2)  circle[radius=.1];
    			\fill[color=red] (0, -2)  circle[radius=.1];
    		\end{tikzpicture}
		}
	\end{figure}
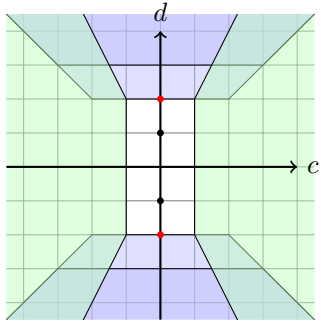
		
	\begin{lemme}
		For all $k \geq 8$, $l \in A$ and $t \in S_{l}$, we have
		\[
			\abs{t - t_k} \leq \frac{k}{2} + 2\sqrt{k} \qquad \text{ and } \qquad
		\]
		\[
			\abs{\Im(t - t_k)} \leq \left\{ \begin{array}{ll}
									\frac{3}{2} \sqrt{k} & \quad \text{ if } t \not\in \set{ \gamma + \beta i }{ i \in \{0, 1, ..., k-1 \} } \\
									2\sqrt{k} & \quad \text{ otherwise. }
								\end{array} \right.
		\]
	\end{lemme}
	
	\begin{proof}
		For every $(i,j) \in \{0,1,...,k-1\}^2$, we have
		$\abs{i + \beta j - t_k} \leq \frac{k}{2} + \frac{3}{2}\sqrt{k}$,
		$\abs{k + \beta i - t_k} \leq \frac{k}{2} + \frac{3}{2}\sqrt{k}$,
		$\abs{\gamma + \beta i - t_k} \leq \sqrt{k + \frac{2}{k}} + \frac{1}{\sqrt{k}} + 1 + \frac{k}{2} + \frac{\sqrt{k}}{2}$ (because the imaginary part of $\beta$ is negative),
		$\abs{k \beta - t_k} \leq \frac{k}{2} + \frac{3}{2}\sqrt{k}$ and
		$\abs{\gamma \beta - t_k} = \abs{\beta^2 - 1 - \frac{k}{2} + \frac{\sqrt{k}}{2} I} \leq \frac{1}{k} + 1 + \frac{k}{2} + \frac{\sqrt{k}}{2}$. Hence, the first inequality is true for $k \geq 8$.
		
		We have
		$\abs{\Im(i + \beta j - t_k)} = \abs{\Im(k + \beta j - t_k)} = \abs{j\Im(\beta) - \frac{\sqrt{k}}{2}} \leq \frac{3}{2} \sqrt{k}$,
		$\abs{\Im(\gamma + \beta i - t_k)} \leq \abs{\Im(\gamma) + \frac{\sqrt{k}}{2}} + i\abs{\beta} \leq \sqrt{k + \frac{2}{k}} + \frac{1}{\sqrt{k}} - \frac{\sqrt{k}}{2} + \sqrt{k}$,
		$\abs{\Im(k\beta - t_k)} \leq k \abs{\beta} + \frac{\sqrt{k}}{2} \leq \frac{3}{2} \sqrt{k}$,
		$\abs{\Im(\gamma \beta - t_k)} \leq \abs{\gamma \beta} + \frac{\sqrt{k}}{2} \leq \frac{\sqrt{k}}{2} + 1 + \frac{1}{k}$. Hence, we get the wanted inequality for $k \geq 3$.
	\end{proof}
	
	\begin{lemme} \label{lcp}
		For all $k \geq 69$, we have $t_k \not\in (\overline{\pi(D_u)} + t_{0,1}) \cup (\overline{\pi(D_u)} + t_{0,-1})$,
		and we have $t_k \not\in \bigcup_{d \in \{ -2,-1,0,1,2\}} (\overline{\pi(D_{u,c})} + t_{0,d})$.
	\end{lemme}
	
	\begin{proof}
		For all $(i,j) \in \{0, 1, ..., k\}$, we have
		$\abs{i + \beta j + t_{0, \pm 1} - t_k} \geq \abs{\Im(\beta j + t_{0, \pm 1} - t_k)}
									= \abs{\Im(\beta)(j \mp k) \pm \Im(\beta^2) + \frac{\sqrt{k}}{2}}$.
		If $\pm = +$, we have $\abs{i + \beta j + t_{0, 1} - t_k} \geq \frac{\sqrt{k}}{2} - \frac{1}{k}$ because $\Im(\beta) < 0$. This is greater than $\frac{1}{1 - \frac{1}{\sqrt{k}}}$ for $k \geq 10$.
		
		If $\pm = -$, we have $\abs{i + \beta j + t_{0, -1} - t_k} \geq \frac{k}{\sqrt{k + \frac{2}{k}}} - \frac{\sqrt{k}}{2} - 1 - \frac{1}{k}$.
		This is greater than $\frac{1}{1 - \frac{1}{\sqrt{k}}}$ for $k \geq 22$.
		
		For $\abs{d} \leq 1$, we have
		$\abs{\gamma \beta + t_{0,\pm 1} - t_k} 	\geq \frac{k}{2} - \frac{\sqrt{k}}{2} - \abs{\beta^2 - 1} - (\frac{1}{k} + \sqrt{k} + 2)
										\geq \frac{k}{2} - 3\frac{\sqrt{k}}{2} - 3 - \frac{2}{k}$.
		This is greater than $\frac{1}{1 - \frac{1}{\sqrt{k}}}$ for $k \geq 24$.
		
		For all $i \in \{0, 1, ..., k-1\}$, and $\abs{d} \leq 2$, we have
		$\abs{\gamma + \beta i + t_{0,d} - t_k} \geq \frac{k}{2} - \sqrt{k + \frac{2}{k}} - \frac{1}{\sqrt{k}} - \sqrt{k} - \abs{t_{0,d}} - \frac{\sqrt{k}}{2} \geq \frac{k}{2} - \frac{1}{\sqrt{k}} - \frac{3}{2}\sqrt{k} - \abs{d}\frac{1}{k} - \abs{d}\sqrt{k} - 2 \abs{d} = \frac{k}{2} - \frac{1}{\sqrt{k}} - \frac{7}{2}\sqrt{k} - \cfrac{2}{k} - 4$. This is greater than $\cfrac{1}{1 - \frac{1}{\sqrt{k}}}$ for $k \geq 69$.
		
		For $\abs{d} \leq 2$, we have
		$\abs{k \beta + t_{0,d} - t_k} \geq \frac{k}{2} - \frac{3}{2}\sqrt{k} - \abs{d}(\frac{1}{k} + \sqrt{k} + 2) \geq \frac{k}{2} - \frac{7}{2} \sqrt{k} - \frac{2}{k} - 4$. This is greater than $\cfrac{1}{1 - \frac{1}{\sqrt{k}}}$ for $k \geq 69$.
	\end{proof}

	Using the lemma~\ref{la} and~\ref{li}, we have that for all the cases not covered by the lemma~\ref{lcp}
	\[
		\abs{t_{c,d}} \geq k - 3\sqrt{k} \quad \text{ or } \quad \abs{\Im(t_{c,d})} \geq 3 \sqrt{k} - 5 \quad \text{ or } \quad \abs{\Im(t_{c,d})} \geq 2 \sqrt{k} - 3.
	\]
	Hence, for all $l \in A$ and all $t \in S_l$, we have
	\[
		\abs{t + t_{c,d} - t_k} \geq k - 3\sqrt{k} - \frac{\sqrt{k}}{2} - \abs{t - \frac{k}{2}} \geq \frac{k}{2} - \frac{9}{2}\sqrt{k} - \sqrt{k + \frac{2}{k}} - \frac{1}{k} \quad \text{ or }
	\]
	\[
		\abs{t + t_{c,d} - t_k} \geq 3 \sqrt{k} - 5 - \frac{\sqrt{k}}{2} - \abs{\Im(t)}
						\geq \frac{3}{2}\sqrt{k} - 5 - \sqrt{k + \frac{2}{k}} - \frac{1}{k} \quad \text{ or }
	\]
	\[
		\abs{t + t_{c,d} - t_k} \geq 2 \sqrt{k} - 3 - \frac{\sqrt{k}}{2} - \abs{\Im(t)}
						\geq \frac{1}{2}\sqrt{k} - 3 \quad \text{ if } t \not\in \set{ \gamma + i\beta }{ i \in \{ 0, 1, ..., k-1 \} }.
	\]
	This is greater than $\cfrac{1}{1 - \frac{1}{\sqrt{k}}}$ for $k \geq 126$ in the first case, for $k \geq 149$ in the second case and for $k \geq 69$ in the third case.
	
	Consequently, we have proven that for every $k \geq 149$, we have $\pi(\overset{\circ}{D}_{u,a}) \neq \emptyset$ because $t_k$ is not in the closure of of $\pi(\Z^A \backslash D_{u,a})$.
	By the theorem~\ref{thm_int}, we obtain the conclusion.
	
	For $0 \leq k < 149$, we can check by computer, using what is done in the section~\ref{sci}, that the interior of $D_{u,a}$ is non-empty, by computing explicitly a regular language describing this interior and checking that this language is non-empty.
	
%
\end{proof}

\begin{rem}
    When we compute the interior of $D_{u,a}$ for these substitutions $s_k$, it appears that we get automata of the same shape for $k$ large enough.
\end{rem}

\begin{conj}
	For all $k \geq 4$, the minimal automaton of the regular language
	$\overset{\circ}{\transp{L}^{s_k}_{a,a}} := \set{u \in \transp{L^{s_k}_{a, a}}}{Q_u \in \overset{\circ}{D}_{u,a}}$
	has $45$ states.
\end{conj}

\begin{rem}
    We prove a similar conjecture for another family of substitutions and we use it to prove the pure discreteness in the next subsection.
\end{rem}

\begin{figure}[h]
	\centering
	\caption{$\pi(D_{u,a} \backslash \overset{\circ}{D}_{u,a})$ for $s_{20}$} \label{fig_bord_Dua} 
	\tohide{
	    \includegraphics[scale=.15]{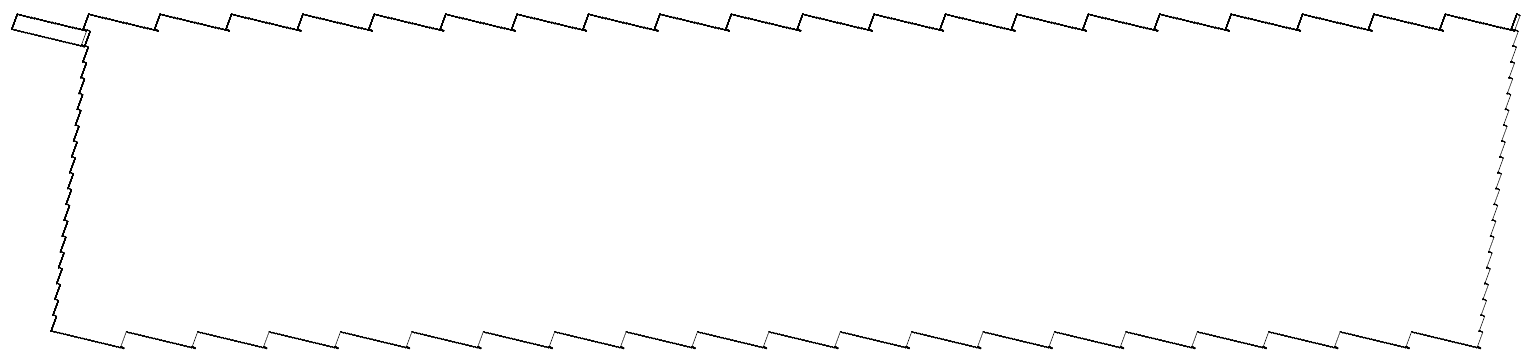}
	}
\end{figure}

\subsection{Proof of pure discreteness using automata}

In this subsection, we prove the pure discreteness using completely different technics
but still as a corollary of theorem~\ref{thm_int}, for an another infinite familly of substitutions:
\[
	s_{l,k} : \left\{ \begin{array}{ccl}
				a &\mapsto& a^l b a^{k-l} \\
				b &\mapsto& c \\
				c &\mapsto& a
			\end{array} \right.
			\quad \forall\ 0 \leq l \leq k,
\]
using the pure discreteness for the substitution
\[
	s_k : \left\{ \begin{array}{ccl}
				a &\mapsto& a^k b \\
				b &\mapsto& c \\
				c &\mapsto& a
			\end{array} \right.
			\quad \forall\ k \in \N_{\geq 1}.
\]
This last substitution is a $\beta$-substitution, and the associated symbolic system is pure discrete after~\cite{barge}, but a similar argument as the one of the previous subsection can also be used to prove the pure discreteness for this family.
We use this fact to prove the pure discreteness for the other family of substitutions.

The idea is to show that the part corresponding to letter $a$ of the discrete line associated to $s_{l,k}$ contains a homothetic copy of the one for $s_k$. More precisely, we prove that $M^2 D_{u,a} \subseteq D_{v,a}$, where $M$ is the incidence matrix of $s_{l,k}$ (it doesn't depends on $l$ and $k$), $v$ is the infinite fixed point of $s_{l,k}$ and $u$ is the infinite fixed point of $s_k$. Hence, we have $\overset{\circ}{D_{u, a}} \neq \emptyset \Longrightarrow \overset{\circ}{D_{v, a}} \neq \emptyset$, and we can use the theorem~\ref{thm_int}.

\begin{rem} \label{rem_wr}
    The symbolic system associated to $s_{l,k}$ is conjugate to the one associated to $s_{k-l, k}$ by word-reversal. Therefore, we can assume without loss of generality that $1 \leq l \leq \ceil{\cfrac{k}{2}}$.
\end{rem}

\subsubsection{Description of $D_{u,a}$ and $D_{v,a}$}

In all the following, $u$ is the infinite fixed point of $s_k$, and $v$ is the infinite fixed point of $s_{l,k}$, for integers $l,k \in \N$, with $k \geq l \geq 1$.
We consider the following map
\[
    \varphi: \left\{\begin{array}{ccc}
                 \Z^A                       &\to        & \Q(\beta)  \\
                 (v_l)_{l \in A}    & \mapsto   & v_a + (\beta-k)v_b + (\beta^2 - \beta k)v_c.
            \end{array}\right.,
\]
where $A = \{a,b,c\}$.
This linear map is one-to-one and has the property that the multiplication by the incidence matrix $M$ in $\R^A$ becomes a multiplication by $\beta$ in $\Q(\beta)$, because $(1, \beta-k, \beta^2 - \beta k)$ is a left eigenvector of $M$ for the eigenvalue $\beta$.
For a language over an alphabet $\Sigma \subseteq \Q(\beta)$, we denote
\[
    Q_{L} := \varphi(Q_{\varphi^{-1}(L)}) = \set{ \sum_{i=0}^{\abs{u}} u_i \beta^i }{ u \in L }.
\]
Using the proposition~\ref{prop_DQ}, we have $\varphi(D_{u,a}) = Q_{L_k}$ and $\varphi(D_{v,a}) = Q_{L_{l,k}}$ where $L_k$ is defined in Figure~\ref{fig_Lk} and $L_{l,k}$ is the regular language defined on Figure~\ref{fig_Llk}, for $\beta$ root of the polynomial $X^3 - kX^2 - 1$.

\begin{figure}[H]
	\centering
	\caption{Automaton defining a language $L_k$ such that $\varphi(D_{u,a}) = Q_{L_k}$, where a transition labeled by $e$ means that there are $k-1$ transitions labeled by $ 1, 2, ..., k-1 $} \label{fig_Lk}
	\tohide{
	    \includegraphics[scale=.6]{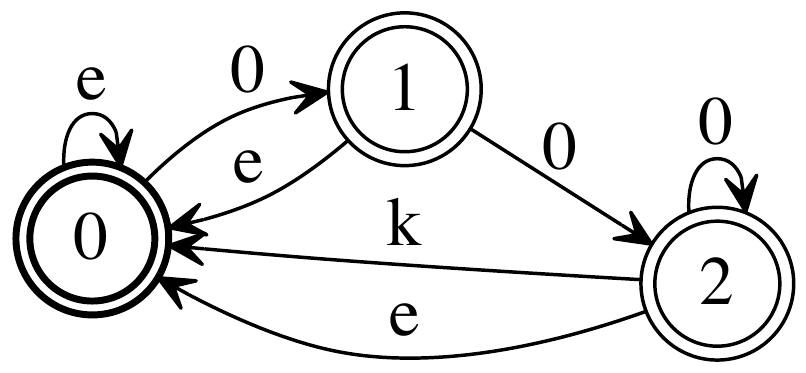}
	}
\end{figure}

\begin{figure}[H]
	\centering
	\caption{Automaton defining a language $L_{l,k}$ such that $\varphi(D_{v,a}) = Q_{L_{l,k}}$, where a transition labeled by $f$ means that there are $l-1$ transitions labeled by $1, 2, ..., l-1$ and $k-l$ transitions labeled by $\beta - k +l, \beta - k + l +1, ..., \beta-2, \beta-1$.} \label{fig_Llk}
	\tohide{
	    \includegraphics[scale=.6]{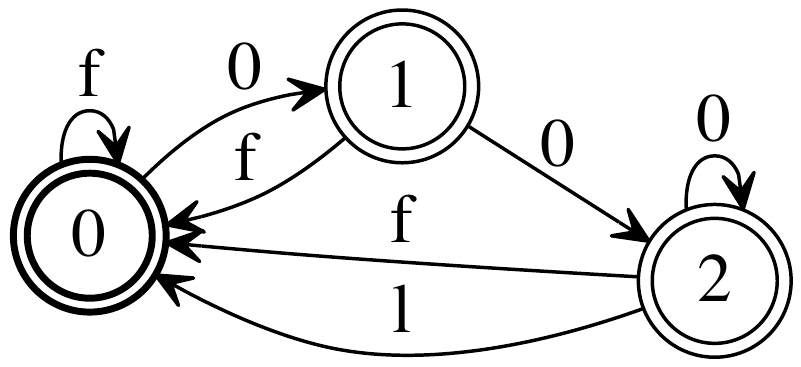}
	}
\end{figure}

\subsubsection{Zero-automaton}

In order to show that $M^2 D_{u,a} \subseteq D_{v,a}$, we need a way to go from the language $L_{l,k}$ to the language $L_k$.
The following proposition permits to do it by describing algebraic relations between a word over the alphabet of $L_k$ and a word over the alphabet of $L_{l,k}$.
It works for $1 \leq l \leq k-2$, but by the remarks~\ref{rem_wr} we can assume it without loss of generality as soon as $k \geq 4$.

\begin{prop} \label{prop_L0}
	Let $L_0$ be the language defined in Figure~\ref{fig_0L}.
	We have
	\[
		 L_0 \subseteq \set{ u \in (\Sigma_k - \Sigma_{l,k})^* }{ \sum_{i=0}^{\abs{u}} u_i \beta^{i} = 0},
	\]
	if $1 \leq l \leq k-2$, where
		$\Sigma_k = \{ 0, 1, ..., k \}$ is the alphabet of the language $L_k$, and
		$\Sigma_{l,k} = \{ 0, 1, ..., l, \beta - k + l, \beta - k + l+1, ..., \beta - 1 \}$ is the alphabet of the language $L_{l,k}$.
\end{prop}

\begin{proof}
    
	\begin{figure}[h]
		\centering
		\caption{Automaton recognizing a language $L_0'$} \label{fig_0Lt}
		\tohide{
		    \includegraphics[scale=.6]{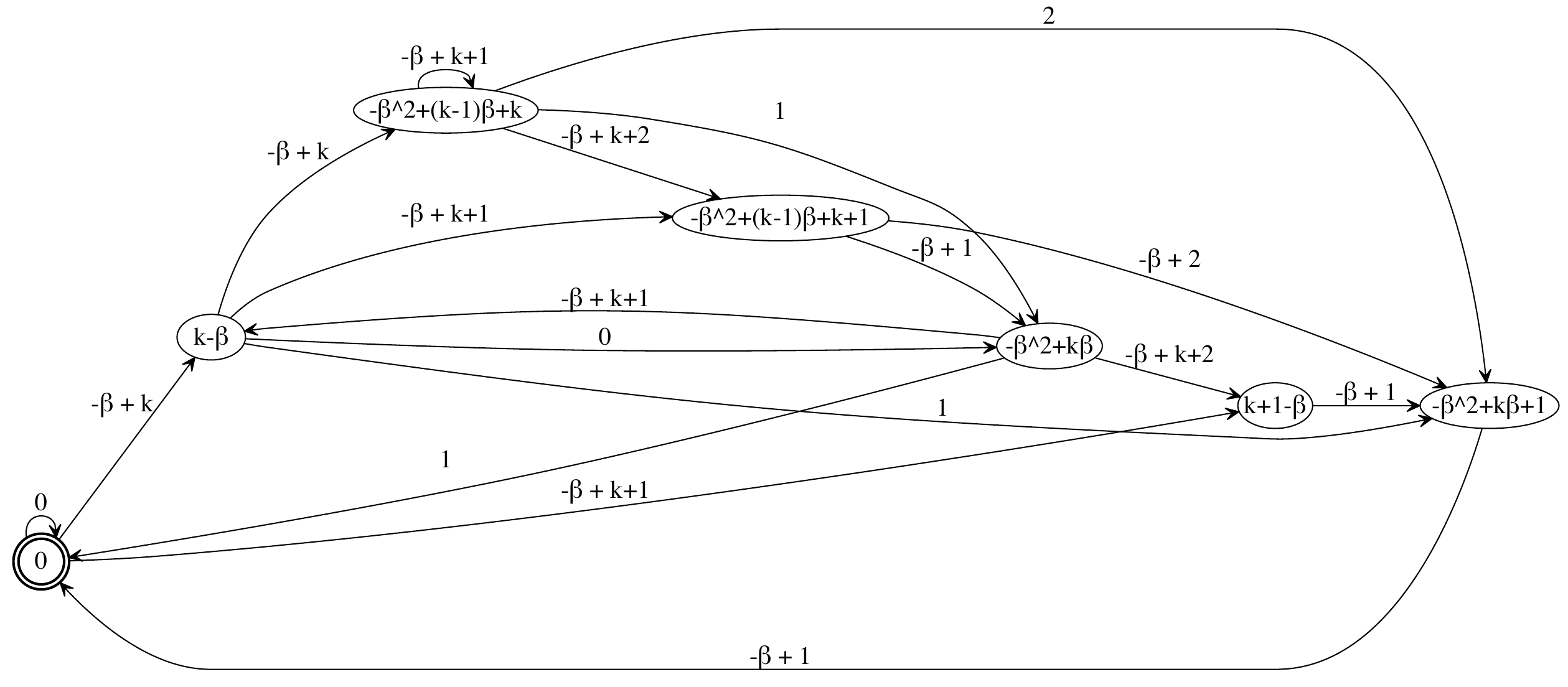}
		}
	\end{figure}

    Let $L_0'$ be the language of the automaton depicted in the picture~\ref{fig_0Lt}.
    We verify easily that $L_0'$ is the transposed (i.e. the word reversal) of the language $L_0$.
    And we easily check that the transitions of the automaton of Figure~\ref{fig_0Lt}, satisfy the following.
	\[
		x \xrightarrow[]{\ t\ } y  \qquad \Longrightarrow \qquad y = \beta x + t, \quad t \in \Sigma_k - \Sigma_{l,k}.
	\]
	Hence, if we have a word $u_0u_1u_2...u_n \in L_0'$, it corresponds to a path from $0$ to $0$, so we have
	\[
	    0 \xrightarrow[]{u_0} u_0 \xrightarrow[]{u_1} \beta u_0 + u_1 \xrightarrow[]{u_2} ... \xrightarrow[]{u_{n-1}} \sum_{i=0}^{n-1} \beta^{n-1-i} u_i \xrightarrow[]{u_n} \sum_{i=0}^n \beta^{n-i} u_i = 0.
	\]
\end{proof}

\begin{figure}[h]
	\centering
	\caption{Automaton $\A_0$ recognizing a language $L_0$} \label{fig_0L}
	\tohide{
	    \includegraphics[scale=.5]{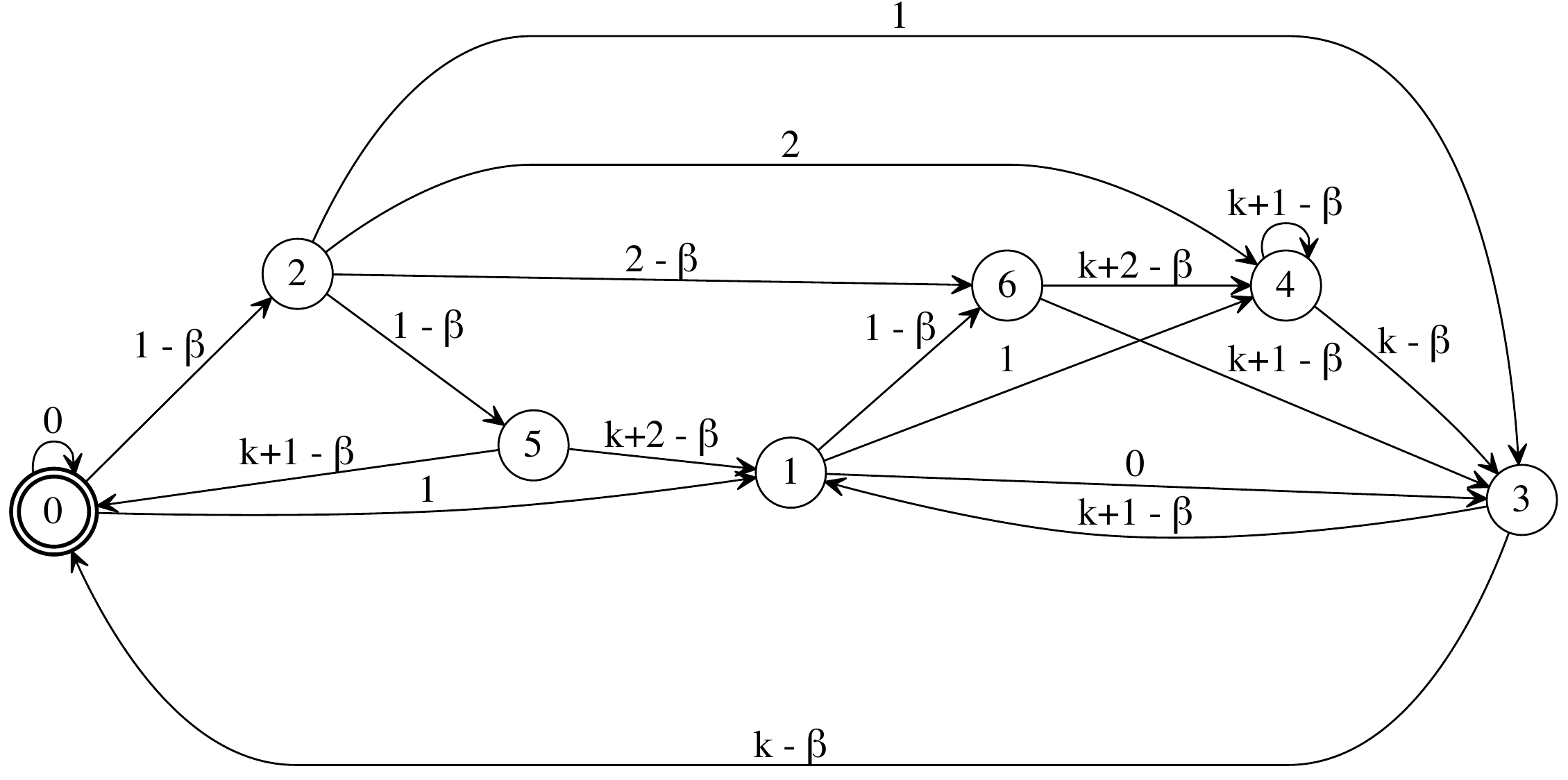}
	}
\end{figure}



\subsubsection{Proof that $M^2 D_{u,a} \subseteq D_{v,a}$}

We define a language $L$ by the automaton $\A$ of Figure~\ref{fig_L}.

\begin{figure}[H]
	\centering
	\caption{The automaton $\A$, recognizing a language $L$. \\ \{0\} is the initial state and every state is final} \label{fig_L} 
	\tohide{
	    \includegraphics[scale=.45]{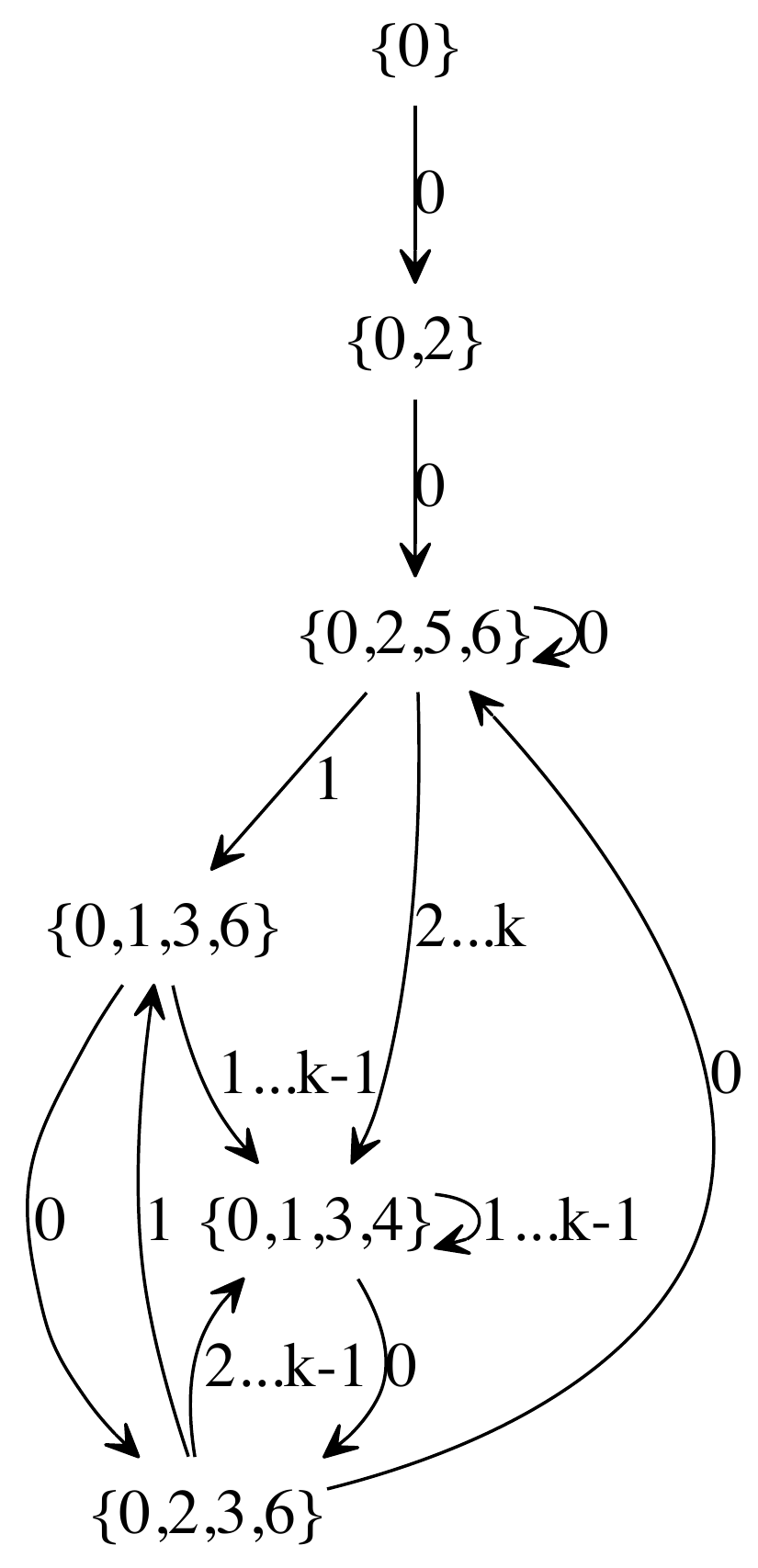}
	}
\end{figure}

\begin{lemme} \label{l_L}
	The transitions of the automaton $\A$ of Figure~\ref{fig_L} satisfy
	\[
		X \xrightarrow[]{\ b\ } Y \quad \Longrightarrow \quad Y \subseteq \set{ y \in Q_0 }{ \exists x \in X,\ \exists c \in \Sigma_{l,k},\ x \xrightarrow[]{\ b - c \ } y \in \A_0 }
	\]
	where $Q_0 = \{ 0,1,2,3,4,5,6\}$ is the set of states of the automaton $\A_0$ of the Figure~\ref{fig_0L}.
\end{lemme}
\begin{proof}
	We can check that using the following array: a star or a letter $l$ means that the set $a \cap (b - \Sigma_{l,k})$ is non-empty for a given $(a,b) \in \Sigma_0 \times \Sigma_{k}$ (the converse is false, but we don't need it),
	where $\Sigma_0 = \{ 0,1 ,2, 1-\beta,2-\beta, k-\beta, k+1-\beta, k+2-\beta \}$ is the alphabet of the automaton $\A_0$. 

	\[
		\begin{array}{|c|c|c|c|c|c|c|c|c|}
			\hline
	\text{\backslashbox{$\Sigma_0$}{$\Sigma_k$}}	& 0		& 1		& 2...l-1	& l			& l+1 		& l+2...k-1 			& k \\
							\hline
							0		& *		& *		& *		& l			& 			&					&	\\
							\hline
							1		& 		& *		& *		& *			& l			&					&	\\
							\hline
							2		&		&		& *		& * 			& *			&					&	\\	
							\hline
							1-\beta 	& *		&		&		&			&			&					&	\\	
							\hline
							2-\beta 	& *		& *		&		&			&			&					&	\\
							\hline
							k - \beta	&		&		&		& *			& *			& *					&	\\
							\hline
							k+1 - \beta &		&		&		&			& *			& *					& * \\
							\hline
							k+2 - \beta &		&		&		&			&			& *					& * \\
							\hline
		\end{array}
	\]
\end{proof}

We remark that every state of the automaton $\A$ of Figure~\ref{fig_L} contains $0$.
Hence, if we have a path $\{0\} \xrightarrow[]{b_1} X_1 \xrightarrow[]{b_2} ... \xrightarrow[]{b_n} X_n$ in this automaton,
we have $0 \in X_n$, and by the lemma~\ref{l_L}, we can find $(c_i)_{i=1}^{n} \in \Sigma_{l,k}^n$ such that we have the following path in $\A_0$:
\[
	0 \xrightarrow[]{b_1 - c_1} x_1 \xrightarrow[]{b_2 - c_2} ... \xrightarrow[]{b_n - c_n} 0.
\]
Then, by definition of the automaton $\A_0$, we have
\[
	Q_b = \sum_{i=1}^n b_i \beta^{i} = \sum_{i=1}^n c_i \beta^{i} = Q_c.
\]

But we have something better:
\begin{lemme}
	For every path $\{0\} \xrightarrow[]{b_1} X_1 \xrightarrow[]{b_2} ... \xrightarrow[]{b_n} X_n$ in the automaton $\A$, we can find a sequence $(c_i)_{i=1}^n \in \Sigma_{l,k}^*$ such that $(b_i - c_i)_{i=1}^n$ labels a path from $0$ to $0$ in the automaton $\A_0$, and with the word $c_1c_2...c_n$ in the language $L_{l,k}$. 
\end{lemme}

\begin{proof}
	The language $L_{l,k}$ is the set of words over the alphabet $\Sigma_{l,k}^*$ such that every letter $l$ is preceded by two letters $0$.
	And we can check that in the proof of the lemma~\ref{l_L}, the only place where we need to take $c_i = l$ is when we follow a transition of $\A$ labeled by $l$ or by $l+1$. 
	And this occurs only when we follow an transition labeled by $0$ or $1$ in the automaton $\A_0$.

	The only transitions of $\A$ that needs to take $c_i = l$ are $\{0,2,5,6\} \xrightarrow[]{l} \{0,1,3,4\}$, $\{0,2,5,6\} \xrightarrow[]{l+1} \{0,1,3,4\}$ and $\{0,1,3,6\} \xrightarrow[]{l+1} \{0,1,3,4\}$.
	But we can check that when we reach the state $\{0,2,5,6\}$, we have read at least two zeroes, so we can assume that the $0$ of the state $\{0,2,5,6\}$ has been reached by following the path
	\[
		0 \xrightarrow[]{0 - 0} 0 \xrightarrow[]{0 - 0} 0
	\]
	in $\A_0$. This allows us to consider the transitions $0 \xrightarrow[]{l - l} 0$ and $0 \xrightarrow[]{(l+1) - l} 1$ of $\A_0$ and getting a word $c_1c_2...c_n$ that stays in the language $L_{l,k}$.
	In the same way, we reach the state $\{0,1,3,6\}$ after reading a $0$ and then a $1$, so we can assume that the $1$ in the state $\{0,1,3,6\}$ has been reached by following the path
	\[
		0 \xrightarrow[]{0-0} 0 \xrightarrow[]{1-0} 1
	\]
	in $\A_0$. This allows us to consider the transition $1 \xrightarrow[]{(l+1) - l} 4$ of $\A_0$ and getting a word $c_1c_2...c_n$ that stays in the language $L_{l,k}$.
\end{proof}

We deduce from this lemma and from the equality $Q_b = Q_c$ that we have $Q_b \in Q_{L_{s,k}}$ for every word $b$ in the language $L$.
Hence, we have the inclusion $\beta^2 Q_{L_k} = Q_{0^2 L_k} \subseteq Q_{L} \subseteq Q_{L_{s,k}}$.
So we have $M^2 D_{u,a} \subseteq D_{v,a}$.

By the theorem~\ref{thm_int}, we have for every $1 \leq l \leq k-2$,
\begin{eqnarray*}
    s_k \text{ satisfy the Pisot substitution conjecture} &\Longrightarrow& \overset{\circ}{D_{u,a}} \neq \emptyset \\
                                &\Longrightarrow& \overset{\circ}{D_{v,a}} \neq \emptyset \quad \text{(because $M^2 D_{u,a} \subseteq D_{v,a}$)}\\
                                &\Longrightarrow& s_{l,k} \text{ satisfy the Pisot substitution conjecture}.
\end{eqnarray*}
And it implies that $s_{l,k}$ satisfy the Pisot conjecture for every $0 \leq l \leq k$, $k \geq 4$, up to take the mirror. For $1 \leq k < 4$, there is a finite number of possibilities, and we can check that it also works for each of them, for example by computing the interior of the discrete line $D_{u,a}$.

\section{Pure discreteness for a $\S$-adic system}

    Let
    \[
        \sigma : \left\{ \begin{array}{ccc}
                a &\mapsto& aab \\
                b &\mapsto& c \\
                c &\mapsto& a
            \end{array} \right.
        \quad \text{ and } \quad
        \tau : \left\{ \begin{array}{ccc}
                a &\mapsto& aba \\
                b &\mapsto& c \\
                c &\mapsto& a
            \end{array}\right.
    \]
    be two substitutions over the alphabet $A = \{a,b,c\}$ having the same incidence matrix
    $M = \begin{pmatrix}2 & 0 & 1 \\ 1 & 0 & 0 \\ 0 & 1 & 0 \end{pmatrix}$.
    
    Given an infinite word $s_0s_1... \in \S^\N$, where $\S = \{\sigma, \tau\}$, we define a word $u \in A^\N$ by
    \[
        u = \lim_{n \to \infty} s_{0}s_{1} ... s_{n}(a).
    \]
    Remark that $s_{0}s_{1}...s_{n}(a)$ is a strict prefix of $s_{0}s_{1}...s_{n}s_{n+1}(a)$, so the limit exists.
    
    We have the following
    \begin{thm} \label{thm_sad}
        For every word $s_0s_1... \in \S^\N$, the subshift
        $(\overline{S^\N u}, S)$ is measurably isomorphic to a translation on a torus.
    \end{thm}
    
    The idea of the proof is similar to the one of the previous subsection: we prove that for every sequence $s_0s_1... \in \S^\N$ we have an inclusion of the form
    \[
        t + M^k D_{u_{\sigma},a} \subseteq D_{u,a},
    \]
    for some $t \in \Z^3$ and $k \in \N$, where $u_{\sigma}$ is the infinite fixed point of the substitution $\sigma$, and we use the theorem~\ref{thm1}.
    
    \subsection{Representation of $D_{u,a}$ by an automaton}
    
    Like for fixed points of substitutions, we can represent $D_{u,a}$ by a finite automaton.
    For simplicity, we will consider rather $\psi(D_{u,a}) \subseteq \Q(\beta)$, where $\beta$ is the highest eigenvalue of $M$, and $\psi: \R^A \to \R$ is the linear map such that $\psi(e_a) = 1$, $\psi(e_b) = \beta - 2$ and $\psi(e_c) = \beta^2 - 2\beta$. This map is such that $\psi(M X) = \beta \psi(X)$ for every $X \in \R^3$.
    
    \begin{figure}[H]
    	\centering
    	\caption{Automaton $\A$, recognizing a language $L$} \label{fig_sadL} 
    	\tohide{
    	    \includegraphics[scale=.45]{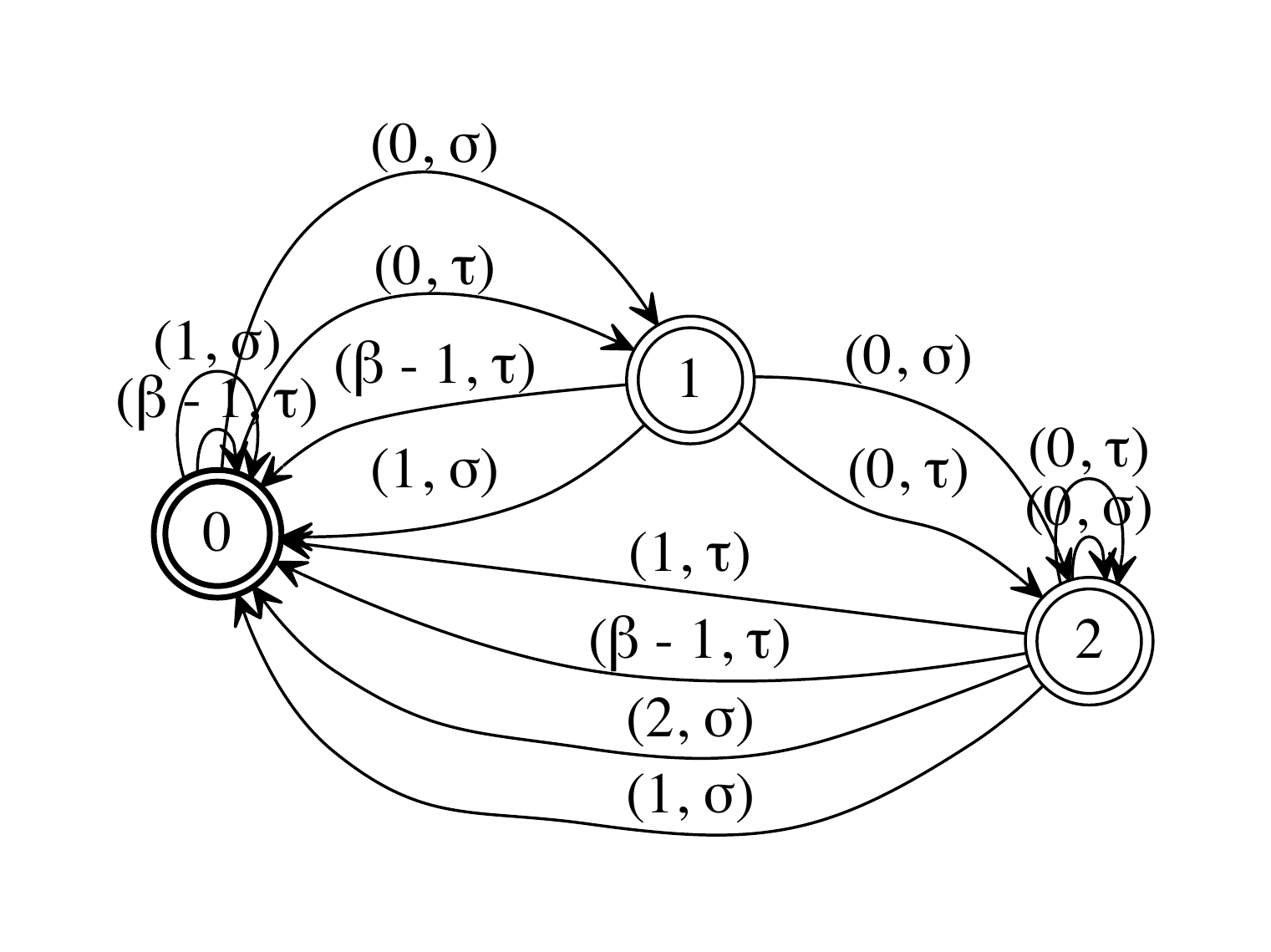}
    	}
    \end{figure}
    
    \begin{prop} \label{p_dsad}
        We have
        \[
            \psi(D_{u,a}) = \set{\sum_{i=0}^{n} u_i \beta^i }{ n \in \N, (u_0, s_0)(u_1, s_1)...(u_n, s_n) \in L },
        \]
        \[
            \psi(D_{u_{\sigma}, a}) = \set{\sum_{i=0}^{n} u_i \beta^i }{ n \in \N, (u_0, \sigma)(u_1, \sigma)...(u_n, \sigma) \in L },
        \]
        where $L$ is the regular language recognized by the automaton of Figure~\ref{fig_sadL}.
    \end{prop}
    
    \begin{proof}
        The proof is very similar to the proof of the proposition~\ref{prop_DQ}.
        We start by constructing a natural map between the prefixes of the word $s_{0} s_{1} ... s_{n}(a) \in A^*$ that are followed by a letter $d$, and the words $v_0 v_1 ... v_n \in \Sigma^*$ of length $n+1$ such that $(v_n, s_n)(v_{n-1}, s_{n-1}) ... (v_1, s_1) (v_0, s_0)$ is in a regular language $L_d$ coming from prefix automata, where $\Sigma = \{ 0, 1, 2, \beta-1 \}$.
        
        \begin{figure}[h]
        	\centering
        	\caption{Automaton $\A'$} \label{fig_Ap}
        	\tohide
        	{
        	    \includegraphics[scale=.49]{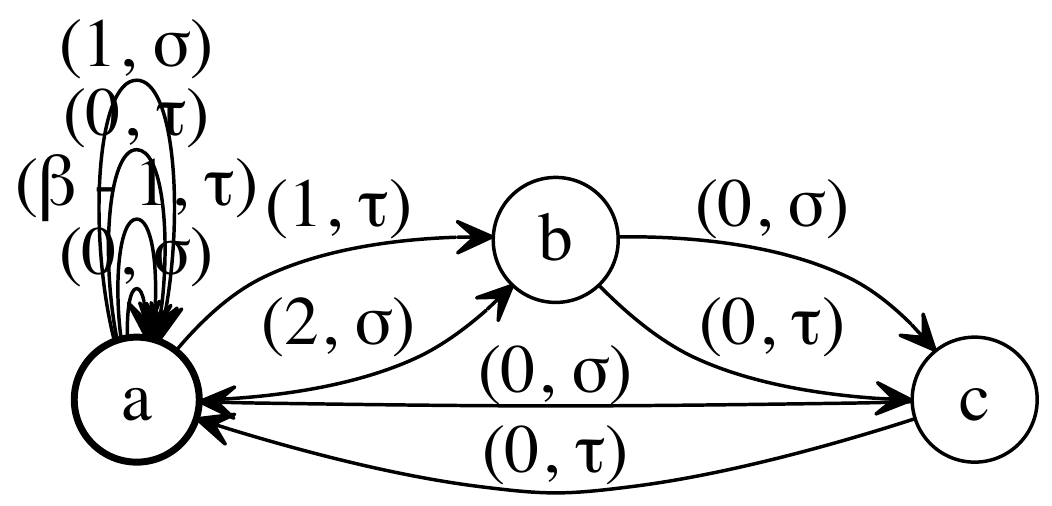}
        	}
        \end{figure}
        
        By combining abelianised prefix automata of the substitution $\sigma$ and $\tau$, we define the automaton $\A'$ of the figure~\ref{fig_Ap}. The following lemma give a direct definition of $\A'$.
        
        \begin{lemme} \label{lAp}
            The automaton $\A'$ of the figure~\ref{fig_Ap} has set of states $A$, and has a transition $d \xrightarrow{(x,t)} e$ if and only if $t \in \S = \{\sigma, \tau \}$, and there exists words $(v', v'') \in A^*$ such that $t(d) = v' e v''$ with $\psi(\Ab(v')) = x$.
        \end{lemme}
        
        \begin{proof}
            Easy verification.
        \end{proof}
        
        \begin{lemme}
            For every $n \in \N$, there exists a natural map
            \[
                \varphi_{s_0s_1...s_n}: \bigcup_{d \in A} \set{v \in A^*}{vd \text{ prefix of } s_0s_1...s_n(a)}
                            \to \bigcup_{d \in A} p_1(L_d^{n+1}),
            \]
            such that
            \[
                \forall d \in A,\ \varphi_{s_0s_1...s_n} \left( \set{v \in A^*}{vd \text{ prefix of } s_0s_1...s_n(a)} \right) = 
                             p_1(L_d^{n+1})
            \]
            where
            \[
                p_1(L_d^{n+1}) = \set{w_n w_{n-1} ... w_1 w_0 \in \Sigma^*}{ (w_n,s_n)(w_{n-1}, s_{n-1}) ... (w_1, s_1)(w_0, s_0) \in L_d },
            \]
            where $L_d$ is the regular language of the automaton $\A'$ of the figure~\ref{fig_Ap}, with initial state $a$ and final state $d \in A$.
        \end{lemme}
        
        \begin{proof}
            By induction on the length of the word $s_0s_1...s_n \in \S^*$. The map $\varphi_{\epsilon}$ is uniquely defined.
            Let $s_0s_1...s_n \in \S^*$ be a word of length at least one, $d \in A$, and $v \in A^*$ such that $vd$ is a prefix of $s_0s_1...s_n(a)$. Then, there exists an unique uplet $(v', v'', v''') \in A^*$ and an unique letter $e \in A$ such that $vdv''' = s_0(v'e)$, and $s_0(e) = v'' d v'''$.
            We define $\varphi_{s_0 s_1 ... s_n}(v) := \varphi_{s_1 s_2 ... s_n}(v') \psi(\Ab(v''))$. By induction hypothesis, we have that $(\varphi_{s_1 s_2 ... s_n}(v'), s_n s_{n-1} ... s_2 s_1) \in L_e$, and by the lemma~\ref{lAp} there exists a transition from state $e$ to state $d$ labeled by $(\psi(\Ab(v'')), s_0)$ in the automaton $\A'$, so we get that $(\varphi_{s_0 s_1 ... s_n}(v), s_n s_{n-1} ... s_1 s_0) \in L_d$ (we identify $\Sigma^* \times \S^*$ with $(\Sigma \times \S)^*$).
        \end{proof}
        
        \begin{lemme}
            For every word $v \in A^*$ and every letter $d \in A$ such that $vd$ is a prefix of $s_0s_1 ... s_n(a)$, we have
            \[
                \psi(\Ab(v)) = \sum_{i=0}^n w_i \beta^i,
            \]
            where $w_n w_{n-1} ... w_1 w_0 = \varphi_{s_0 s_1 ... s_n}(v)$.
        \end{lemme}
        
        \begin{proof}
            By construction of $\varphi_{s_0s_1 ... s_n}(v)$, there exists a sequence of letters $d_0 d_1 ... d_n d_{n+1} \in A^{n+2}$ and two sequences of words $v_0 v_1 ... v_n$ and $w_0 w_1 ... w_n \in (A^*)^{n+1}$ such that
            $\forall\ 0 \leq k \leq n$, $s_{k}(d_{k+1}) = v_{k} d_{k} w_{k}$, with $d_{n+1} = a$ and $d_{0} = d$.
            Then, we have
            \[
                s_0 s_1 ... s_n(a) = (s_0 s_1 ... s_{n-1})(v_{n}) (s_0 s_1 ... s_{n-2})(v_{n-1}) ... s_0(v_{1}) v_0 d_0 w_0 s_0(w_1) ...
            \]
            \[
                ... (s_0 s_1 ... s_{n-2})(w_{n-1}) (s_0 s_1 ... s_{n-1})(w_n),
            \]
            and
            \[
                v = (s_0 s_1 ... s_{n-1})(v_{n}) (s_0 s_1 ... s_{n-2})(v_{n-1}) ... s_0(v_{1}) v_0.
            \]
            Hence, we get
            \[
                \Ab(v) = (M_0 M_1 ... M_{n-1}) \Ab(v_n) + (M_0 M_1 ... M_{n-2}) \Ab(v_{n-1}) + ... + M_0 \Ab(v_1) + \Ab(v_0),
            \]
            where $M_i = M_{s_i}$ is the matrix of the substitution $s_i$.
            But, here we have $M_i = M$ for all $i \in \N$, because $M_\sigma = M_\tau = M$, so we get
            \[
                \psi(\Ab(v)) = \beta^n w_n + \beta^{n-1} w_{n-1} + ... + \beta w_1 + w_0,
            \]
            where $w_i = \psi(\Ab(v_i))$. And we have $\varphi_{s_0 s_1 ... s_n}(v) = w_n w_{n-1} ... w_1 w_0$ by definition of $\varphi_{s_0 s_1 ... s_n}$.
        \end{proof}
        
        \begin{lemme}
            The language $L$ is the mirror of the language $L_a$.
        \end{lemme}
        
        \begin{proof}
            Easy verification. We get the automaton $\A$ from the automaton $\A'$ with initial state $a$ and final state $a$, by reversing the transitions, and then using an usual algorithm, called power set construction, to compute a deterministic automaton from it. With this construction, the state $0$ of $\A$ corresponds to $\{a\}$, the state $1$ corresponds to $\{a,b\}$, and the state $2$ corresponds to $\{a,b,c\}$.
        \end{proof}
        
        These lemma give a proof of the proposition~\ref{p_dsad}, because we have
        \begin{eqnarray*}
            \psi(D_{u,a})   &=& \psi(\Ab(\bigcup_{n \in \N} \set{v \in A^*}{va \text{ prefix of } s_0s_1...s_n(a)})) \\
                            &=& \bigcup_{n \in \N} \{\sum_{i=0}^n w_i \beta^i\ |\  w_0w_1 ... w_n = \varphi_{s_0s_1...s_n}(v),\ \text{ where } v \in \Sigma^*,\\
                            && \qquad \qquad \qquad \qquad \text{with } va \text{ prefix of } s_0s_1...s_n(a)\} \\
                            &=& \set{\sum_{i=0}^n w_i \beta^i}{ (w_n, s_n) (w_{n-1}, s_{n-1}) ... (w_1, s_1) (w_0, s_0) \in L_a} \\
                            &=& \set{\sum_{i=0}^n w_i \beta^i}{ (w_0, s_0) (w_1, s_1) ... (w_{n-1}, s_{n-1}) (w_n, s_n) \in L }.
        \end{eqnarray*}
        And the second equality of the proposition is a particular case of the first one, where we take for all $i \in \N$, $s_i = \sigma$.
        
    \end{proof}
    
    Now that we have a description of the discrete lines $D_{u,a}$ and $D_{u_{\sigma},a}$, we use it to show that for every word $s_0s_1... \in \S^\N$ we have an inclusion of the form
    \[
        M^k D_{u_{\sigma},a} + t \subseteq D_{u,a}.
    \]
    
    \subsection{Proof of the inclusion} \label{ss_pi}
    
    Let $\Sigma_\sigma = \{0,1,2\}$ and $\Sigma_\tau = \{0,1,\beta-1\}$.
    We define a regular language $L_*$ over the alphabet $\Sigma_\sigma \times \S$ by
    \[
        L_* = (\Sigma_\sigma \times \S)^* \cap m(L_0 \times L_\sigma \times L),
    \]
    where
    \[
        L_\sigma = \set{u_0u_1...u_n \in \Sigma_\sigma^*}{n \in \N, (u_0,\sigma)(u_1,\sigma)...(u_n,\sigma) \in L}
    \]
    \[
        L_0 = \set{u_0u_1...u_n \in \Sigma^*}{n \in \N, \sum_{i=0}^n u_i \beta^i = 0 }
    \]
    \[
        \Sigma' = \Sigma_\sigma - \Sigma_\tau = \{ -1, 0, 1, 2, 1-\beta, 2-\beta, 3-\beta \}
    \]
    and $m$ is the word morphism defined by
    \[
        m: \begin{array}{ccc}
                    \Sigma' \times \Sigma_\sigma \times \Sigma_L &\to& \Sigma_\sigma \times \S \cup \{*\} \\
                    (t,x,(y,i)) &\mapsto& \left\{ \begin{array}{cc}
                         (x,i) & \text{ if } x-y = t \\
                          * & \text{ otherwise}
                    \end{array} \right.
                \end{array}
    \]
    where $\Sigma_L = (\Sigma_\sigma \cup \Sigma_\tau) \times \S$ is the alphabet of the language $L$.
    
    \begin{lemme}
        We have
        \[
            \psi(D_{u,a}) \supseteq \set{\sum_{i=0}^{n} u_i \beta^i }{ n \in \N, (u_0, s_0)(u_1, s_1)...(u_n, s_n) \in L_*}.
        \]
    \end{lemme}
    
    \begin{proof}
        For all $n \in \N$, we have
        \begin{eqnarray*}
            && (x_0,s_0)(x_1, s_1)...(x_n,s_n) \in L_* \\
            &\Longrightarrow& \exists (y_0,s_0)(y_1,s_1)...(y_n,s_n) \in L,\ \sum_{i=0}^n (x_i - y_i) \beta^i = 0 \\
            &\Longrightarrow& \sum_{i=0}^n x_i \beta^i \in \psi(D_{u,a}).
        \end{eqnarray*}
    \end{proof}
    
    \begin{figure}[H]
    	\centering
    	\caption{Automaton recognizing a language $L_u$} \label{fig_li} 
    	\tohide{
    	    \includegraphics[scale=.45]{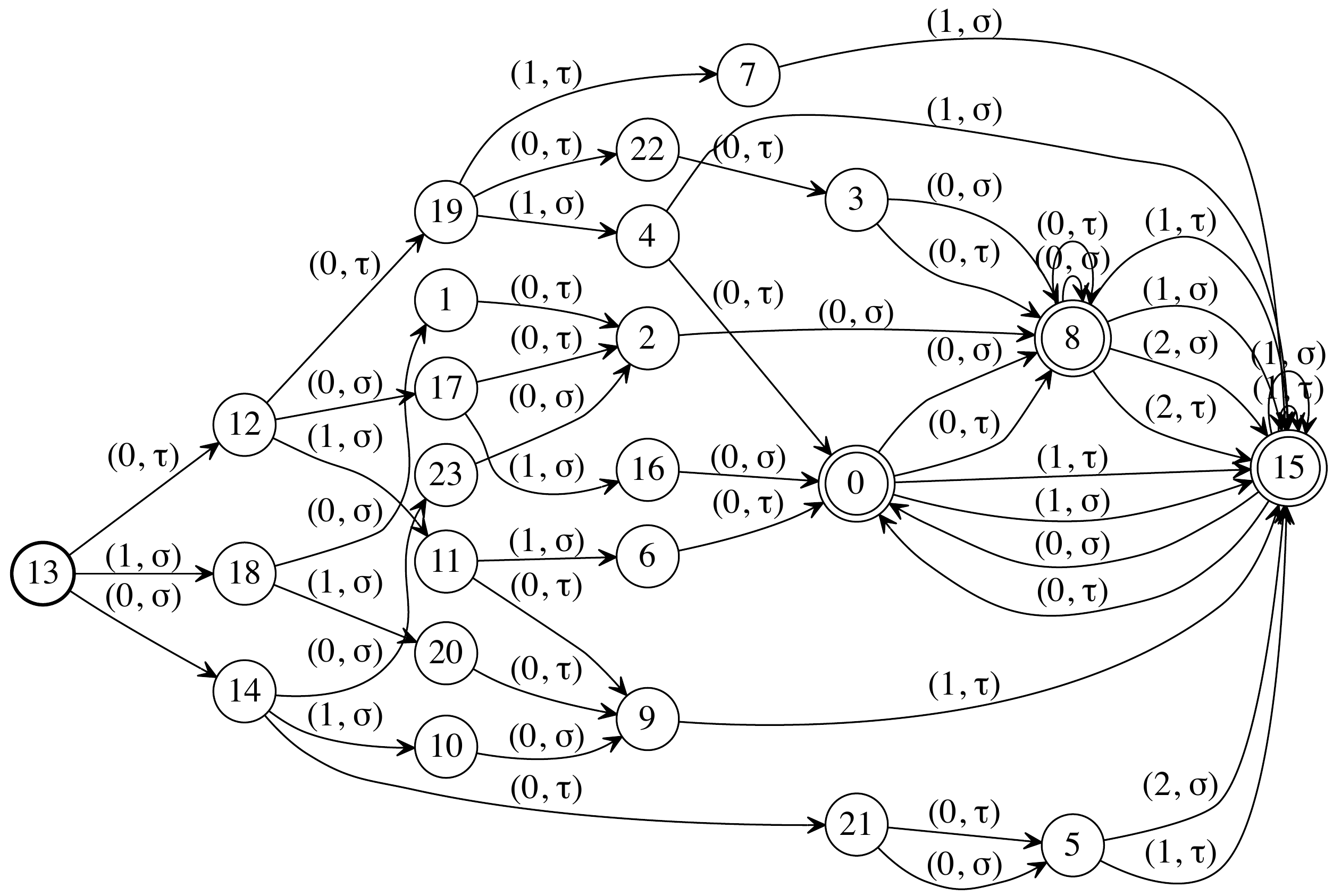}
    	}
    \end{figure}
    
    \begin{lemme}
        We have
        \[
            L_* \supseteq L_u,
        \]
        where $L_u$ is the language defined in Figure~\ref{fig_li}.
    \end{lemme}
    
    \begin{proof}
        Computation done by computer.
        The language $L_0$ is regular thanks to~\cite{me}, and its minimal automaton has $62$ states.
        The minimal automaton of the language $L_*$ has $210$ states.
    \end{proof}
    
    We deduce from these two lemma that for every sequence $s_0 s_1... \in \S^\N$, we have an inclusion  of the form $\psi(D_{u,a}) \supseteq t + \beta^k \psi(D_{u_{\sigma}, a})$, for some $t \in \Z[\beta]$ and some $k \in \N$ that depend of the beginning $s_0s_1s_2s_3s_4s_5$ of the $\S$-adic sequence. Thus we have
    \[
        t' + M^k D_{u_{\sigma},a} \subseteq D_{u,a}
    \]
    for some $t' \in \Z^3$.
    
    For example, if $s_0 = s_1 = \sigma$, and $s_2 = s_3 = \tau$, then we have the inclusion
    \[
        e_a + Me_a + M^3e_a + M^4 D_{u_{\sigma}, a} \subseteq D_{u,a}.
    \]
    If we have $s_0 = s_1 = s_2 = s_3 = s_4 = \tau$, then we have the inclusion
    \[
        M^5 D_{u_{\sigma}, a} \subseteq D_{u,a}.
    \]
    
    \subsection{Pure discreteness of the spectrum}
    
    In order to use the theorem~\ref{thm1}, we need the following property.
    
    \begin{lemme}
        The subshift $(\overline{S^\N u}, S)$ is minimal for every sequence $s_0s_1... \in \S^\N$.
    \end{lemme}
    
    \begin{proof}
        In the word $u$, there are at most $5$ letters between two consecutive letters $a$.
        Indeed, $u = s_{0}s_{1}(v)$ for an infinite word $v \in A^\N$, and we have that $s_{0}s_{1}(a)$ is a word of length $7$ with $4$ letters $a$, $s_{0}s_{1}(b) = a$ and $s_{0}s_{1}(c)$ is a word of length $3$ with $2$ letters $a$.
        
        Thus, there exists a constant $C$ such that every factor of $u$ of length $\geq C \beta^n$ contains $s_{0}s_{1}...s_{n}(a)$. Indeed, it suffices to see that $u = s_{0}s_{1}...s_{n}(w)$ with a word $w$ that satisfies the above property.
        
        If a word is in $\overline{S^\N} u$, then it contains arbitrarily large factors of $u$, so it contains $s_{0}s_{1}...s_{n}(a)$ for every $n \in \N$. Therefore this word is dense in $S^\N u$.
    \end{proof}
    
    Let $\sigma_-$ be the Galois morphism
        \[
            \sigma_-:\begin{array}{ccc}
                        \Q(\beta) &\to& \Q(\gamma) \\
                        \beta &\mapsto& \gamma
                    \end{array},
        \]
        where $\gamma$ is a complex conjugate of $\beta$.
    Then, we define $\pi := \sigma_- \circ \psi : \Z^A \to \C$, and this extends uniquely to a linear map $\R^A \to \C$.
    
    \begin{props}
        The map $\pi$ has the following properties:
        \begin{itemize}
            \item $\forall X \in \R^A$,\ $\pi(MX) = \gamma \pi(X)$,
            \item $\ker(\pi)$ is the eigenspace of $M$ for the greatest eigenvalue,
            \item $\pi(\Z^A)$ is dense in $\C$,
            \item $\pi(D_u)$ is bounded.
        \end{itemize}
    \end{props}
    
    \begin{proof}
        We have for every $X \in \Z^A$, $\psi(MX) = \beta X$, thus $\pi(MX) = \gamma X$. Then it is true for every $X \in \R^A$ by linearity.
        If $V$ is a Perron eigenvector of $M$, then $\gamma \pi(V) = \pi(MV) = \pi(\beta V) = \beta \pi(V)$ so $\pi(V) = 0$. The hypothesis that $M$ is irreducible (i.e. the characteristic polynomial is irreducible) implies that the coordinates of a Perron eigenvector $V$ are linearly independent over $\Q$, and this implies that $\pi(\Z^A)$ is dense in $\C$ (see remark~\ref{rem_dense}).
        And it is not difficult to see that the set $\pi(D_u)$ is bounded, by using for example the description of $\psi(D_u)$ given in the proposition~\ref{p_dsad} and the fact that $\abs{\gamma} < 1$.
    \end{proof}
    
    In order to prove the disjointness in measure of the translated copies of $\overline{\pi(D_u)}$ by $\pi(\Gamma_0)$, we use the same strategy than in the proof of the theorem~\ref{thm_int}: we show that the interior of $D_{u,a}$ is non-empty, and we show that the interior of $D_u$ is dense.
    It is known that we have $\overset{\circ}{D}_{u_{s_1},a} \neq \emptyset$ for the topology defined on the subsection~\ref{topo} (we can prove it by computing this interior explicitly thanks to the theorem~\ref{cint}). And the matrix $M$ is in $GL(3, \Z)$, so if we consider an inclusion of the form $t + M^k D_{u_{s_1},a} \subseteq D_{u,a}$, $t \in \Z^3$, $k \in \N$ given by the previous subsection, then it implies that $\overset{\circ}{D_{u,a}} \neq \emptyset$, and this is true for every sequence $s_0s_1... \in \S^\N$.
    Then, we have the following result.
    \begin{lemme} \label{l_sad_dense}
        $\forall l \in \{a,b,c\},\ \overset{\circ}{D_{u,l}}$ is dense in $D_{u,l}$.
    \end{lemme}
    
    \begin{proof}
        Let $u_n = \lim_{k \to \infty} s_{n} s_{n+1}...s_{n+k}(a)$.
        We have just proven that for every $n \in \N$, $D_{u_n,a}$ has non-empty interior.
        And we have $u = s_{0}s_{1}...s_{n-1}(u_{n})$, so we get the equality
        \[
            \psi(D_{u,i}) = \bigcup_{i \xrightarrow{(t_0,s_0)} ... \xrightarrow{(t_{n-1},s_{n-1})} j \in \A} \beta^n \psi(D_{u_n,j}) + \sum_{k=0}^{n-1} t_k \beta^k
        \]
        for all $i,j \in \{a,b,c\}$.
        But the automaton $\A$ is such that we can reach any state from any state, even if we impose the right coefficients of labels read.
        Hence, we can approach (for our topology) any point of $D_{u,i}$ by subsets of $D_{u,i}$ of the form $M^k D_{u_k,a}+t$, $t \in \Z^3$, $k \in \N$. Such subsets have non-empty interior since $M \in GL(3,\Z)$. This ends the proof. 
    \end{proof}
    
    \begin{lemme} \label{l_sad_b}
        The boundary of $\pi(D_u)$ has zero Lebesgue measure.
    \end{lemme}
    
    In order to prove this lemma, let introduce some notations.
    For all $n \in \N$, let $u_n = \lim_{k \to \infty} s_{n} s_{n+1} ... s_{n+k}(a)$, and for all $a \in A$, $R_a^n = \overline{\pi(D_{u_n,a})} \subseteq \C$. We denote by $\lambda$ the Lebesgue measure of $\C$.
    We have the following
    \begin{lemme}
        For every $a \in A$, the sequence $(\lambda(R_a^n))_{n \in \N}$ is increasing and bounded.
    \end{lemme}
    
    \begin{proof}
        By the proposition~\ref{p_dsad}, we have the following equality
        \[
            \psi(D_{u_n,a}) = \bigcup_{b \xrightarrow[]{(t,s_{n+1})} a \in \A} \beta \psi(D_{u_{n+1},b}) + t,
        \]
        where $\A$ is the automaton of the figure~\ref{fig_sadL}.
        Thus, we have the equality
        \[
            R_a^n = \bigcup_{b \xrightarrow[]{(t,s_{n+1})} a \in \A''} \gamma R_b^{n+1} + t,
        \]
        where $\A''$ is the automaton $\A$ where we apply the Galois morphism $\sigma_-$ to the labels of the transistions.

        Then, we have
        \[
            \lambda(R_a^n) \leq \sum_{b \xrightarrow[]{(t,s_{n+1})} a \in \A''} \frac{1}{\beta} \lambda(R_b^{n+1}).
        \]
        If we take the vector $X_n = (\lambda(R_a^{n}))_{a \in A} \in \R^A$, the previous inequality becomes
        \[
            X_n \leq \frac{1}{\beta} M X_{n+1}.
        \]
        But by the Perron-Frobenius theorem, we have the inequality $M X \leq \beta X$ for every $X \in \R_+^A$, so we get that $X_n$ is increasing.
        The coefficient of $X_n$ are also bounded by $\cfrac{\max_{t \in \Sigma''} \abs{t}}{1 - \abs{\gamma}}$, where $\Sigma''$ is the alphabet of the automaton $\A''$.
    \end{proof}
    
    This lemma give the existence of the limit
    $\lambda_a^\infty = \lim_{n \to \infty} \lambda(R_a^n)$.
    We have the following lemma.
    
    \begin{lemme}
        There exists $\epsilon >0$ and $\eta > 0$ such that for every $n \in \N$ and every $a \in A$, there exists $t \in \C$ and $r > 0$ such that the ball $B(t, r + \epsilon)$ is included in $R_a$, and such that $\lambda(B(t, r)) \geq \eta \lambda_a^\infty$.
    \end{lemme}
    
    \begin{proof}
        It is an immediate consequence of the inclusions proven in the subsection~\ref{ss_pi}.
    \end{proof}
    
    \begin{lemme} \label{l_bn}
        There exists $n_0 \in \N$ such that for every $n \geq n_0$ and every $a \in A$, we have $\lambda(\partial R_a^n) = 0$.
    \end{lemme}
    
    \begin{proof}
        Let $n_0 \in \N$ such that
        \[
            \forall a \in A,\ \lambda_a^\infty \leq (1 + \eta) \lambda(R_a^{n_0}),
        \]
        and let $k \in \N$ such that for every $a \in A$, every $(t,t') \in \C^2$, every $r > 0$, and every $n \in \N$,
        \[
            \gamma^k R_a^n + t \cap B(t', r) \neq \emptyset \Longrightarrow \gamma^k R_a^n + t \subseteq B(t', r+\epsilon).
        \]
        
        Let us show that for every $n \geq n_0$ we have
        \[
            \forall a \in A,\ \lambda(\partial R_{a}^{n+k}) \leq c \lambda(R_a^{n+k})
            \Longrightarrow \forall a \in A,\ \lambda(\partial R_{a}^{n}) \leq c(1 - \eta^2) \lambda(R_a^{n}).
        \]
        
        Let $a \in A$ and $n \geq n_0$.
        Let $t \in \C$ and $r > 0$ such that $B(t, r+\epsilon) \subseteq R_a^n$ and $\lambda(B(t, r)) \geq \eta \lambda_a^\infty$.
        Let
        \[
            T_b = \set{ \sum_{j=n}^{n+k-1} \gamma^{n+k-j-1} t_j }{ b \xrightarrow{(t_n,s_n)} ... \xrightarrow{(t_{n+k-1},s_{n+k-1})} a \in \A'},
        \]
        \[
            T'_b = \set{\sum_{j=n}^{n+k-1} \gamma^{n+k-j-1} t_j }{ b \xrightarrow{(t_n,s_n)} ... \xrightarrow{(t_{n+k-1},s_{n+k-1})} a \in \A' \text{ and } (\gamma^k \partial R_b^{n+k} + t) \cap B(t, r) = \emptyset}.
        \]
        Then we have
        \begin{eqnarray*}
            \lambda(\partial R_a^{n}) &\leq& \lambda\left( \bigcup_{b \in A} \bigcup_{t \in T'_b} (\gamma^k \partial R_b^{n+k} + t) \right) \\
                                    &\leq& \sum_{b \in A,\ t \in T'_b} \frac{1}{\beta^k} \lambda(\partial R_b^{n+k} ) \\
                                    &\leq& \frac{c}{\beta^k} \sum_{b \in A,\ t \in T'_b} \lambda(R_b^{n+k}) \\
                                    &\leq& \frac{c}{\beta^k} \left[ \left( \sum_{b \in A,\ t \in T_b} \lambda(R_b^{n+k}) \right) - \beta^k \lambda(B(t,r)) \right] \\
                                    &\leq& \frac{c}{\beta^k} \left( \beta^k \lambda(R_a^{n+k}) - \beta^k \eta \lambda_a^\infty \right) \\
                                    &\leq& c (1-\eta) \lambda_a^\infty \\
                                    &\leq& c (1-\eta^2) \lambda(R_a^n).
        \end{eqnarray*}
        
        We deduce from these equalities that we have
        \[
            \lambda(\partial R_a^n) \leq (1 - \eta^2)^k \lambda(R_a^n) \xrightarrow[k \to \infty]{} 0.
        \]
    \end{proof}
    
    \begin{proof}[Proof of the lemma~\ref{l_sad_b}]
        We have the inclusion
        \[
            \partial R_a \subseteq \bigcup_{b \xrightarrow{(t_0,s_0)} ... \xrightarrow{(t_{n-1},s_{n-1})} a \in \A'} \gamma^n \partial R_{b}^n + \sum_{j=0}^{n-1} \gamma^{n-1-j} t_j. 
        \]
        And by the lemma~\ref{l_bn}, we have $\lambda(\partial R_b^n) = 0$ for $n \geq n_0$ and $b \in A$.
        Thus the boundary of $R_a$ has zero Lebesgue measure.
    \end{proof}
    
    Thanks to the lemma~\ref{l_sad_dense}, for every $t \in \Gamma_0 \backslash \{0\}$, the empty intersection $\overset{\circ}{D_u} \cap \overset{\circ}{D_u} + t$ is a dense open subset of $\overset{\circ}{\overline{D_u}} \cap \overset{\circ}{\overline{D_u}} + t$.
    Hence, the interior of $\overline{\pi(D_u)}$ and $\overline{\pi(D_u + t)}$ are disjoint.
    By the lemma~\ref{l_sad_b}, it proves that the Lebesgue measure of the intersection is zero.
    
    Every hypothesis of the theorem~\ref{thm1} is satisfied, thus the subshift $(\overline{S^\N u}, S, \mu)$ is uniquely ergodic and measurably conjugate to the translation on the torus $(\Part / \pi(\Gamma_0), T, \lambda)$.
    This ends the proof of the theorem~\ref{thm_sad}.
    

\end{document}